\documentclass[11pt]{article}
\usepackage{amsmath,amssymb,amsthm,amsfonts,amstext,amsbsy,amscd}
\usepackage{graphics}
\usepackage{graphicx}

\RequirePackage[colorlinks,citecolor=blue,urlcolor=blue]{hyperref}

\RequirePackage{hypernat}

\def\<{\langle}
\def\>{\rangle}

\def\Chi{\raise .3ex
\hbox{\large $\chi$}} 
\newcommand{\norme}[1]{ {\left\lVert  #1\right\rVert}}

\def\({\Bigl (}
\def\){\Bigr )}
\newcommand{\be}{\begin{equation}}
\newcommand{\ee}{\end{equation}}
\newcommand{\bea}{$$ \begin{array}{lll}}
\newcommand{\eea}{\end{array} $$}
\newcommand{\bi}{\begin{itemize}}
\newcommand{\ei}{\end{itemize}}

\numberwithin{equation}{section}
\newtheorem{satz}{Satz}[section]

\newtheorem{theorem}[satz]{Theorem}

\newtheorem{lemma}[satz]{Lemma}
\newtheorem{assumption}[satz]{Assumption}

\newtheorem{definition}[satz]{Definition}

\newtheorem{remark}[satz]{Remark}

\DeclareMathOperator{\E}{{\mathbb E}}

\DeclareMathOperator{\R}{{\mathbb R}}

\DeclareMathOperator{\PP}{{\mathbb P}}

\DeclareMathOperator{\supp}{supp}

\renewcommand{\phi}{\varphi}
\renewcommand{\cdot}{{\scriptstyle \bullet} }

\renewcommand{\le}{\leqslant}
\renewcommand{\ge}{\geqslant}

\newcommand{\ev}{\operatorname{even}}
\newcommand{\od}{\operatorname{odd}}

\newcommand{\footnoteremember}[2]{
  \footnote{#2}
  \newcounter{#1}
  \setcounter{#1}{\value{footnote}}
}
\newcommand{\footnoterecall}[1]{
  \footnotemark[\value{#1}]
}

\begin{document}
\title{Adaptive wavelet estimation of the diffusion coefficient under additive error measurements}
\author{M. Hoffmann\footnote{ENSAE and CNRS-UMR 8050, 3, avenue Pierre Larousse, 92245 Malakoff Cedex, France.}, \; A. Munk\footnoteremember{adr}{Institut f\"ur Mathematische Stochastik, Universit\"at G\"ottingen, Goldschmidtstr. 7, 37077 G\"ottingen, Germany.} and J. Schmidt-Hieber\footnoterecall{adr}}

\date{}
\maketitle

\begin{center}
{\tt Revised version}
\end{center}

\begin{abstract} 
We study nonparametric estimation of the diffusion coefficient from discrete data, when the observations are blurred by additional noise. Such issues have been developed over the last 10 years in several application fields and in particular in high frequency  financial data modelling, however mainly from a parametric and semiparametric point of view. This paper addresses the nonparametric estimation of the path of the (possibly stochastic) diffusion coefficient in a relatively general setting. 

By developing pre-averaging techniques combined with wavelet thresholding, we construct adaptive estimators that achieve a nearly optimal rate within a large scale of smoothness constraints of Besov type. Since the diffusion coefficient is usually genuinely random, we propose a new criterion to assess the quality of estimation; we retrieve the usual minimax theory when this approach is restricted to a deterministic diffusion coefficient. In particular, we take advantage of recent  results of Rei\ss\, \cite{rei2} of asymptotic equivalence between a Gaussian diffusion with additive noise and Gaussian white noise model, in order to prove a sharp lower bound. 
\end{abstract}

\section{Introduction}
We are interested in the following statistical setting: we assume that we have real-valued data of the form
\begin{equation} \label{observations}
Z_{j,n} = X_{j\Delta_n}+\epsilon_{j,n},\;\;j=0,1,\ldots, n
\end{equation}
where $\Delta_n >0$ is a sampling time, $(\epsilon_{j,n})$ is an additive noise process\footnote{implicitly assumed to be centered for obvious identifiability purposes.} and the continuous time process $X = (X_t)_{t \geq 0}$ has representation
\begin{equation} \label{defdiffusion}
X_t = X_0 + \int_0^t b_s\, ds +\int_0^t\sigma_s\, dW_s,
\end{equation}
In other words, $X$ is an It\^o continuous semimartingale driven by a Brownian motion $W = (W_t)_{t \geq 0}$
with drift $b = (b_t)$ and diffusion coefficient or volatility process $\sigma = (\sigma_t)$. This is the so-called additive microstructure noise model. We assume that the data $(Z_{j,n})$ are sampled in a high-frequency framework: the time step $\Delta_n$ between observations goes to $0$, but $n\Delta_n$ remains bounded as $n \rightarrow \infty$, {\it i.e.} the whole statistical experiment is taken over a fixed time interval. 

In this asymptotic framework, the only parameter that can be consistently estimated is the unobserved path of the diffusion coefficient $t \leadsto \sigma_t^2$, and unless specified otherwise, it is random. 
Whereas nonparametric estimation of the diffusion coefficient from direct observation $X_{j\Delta_n}$ is a fairly well known topic when $\sigma^2$ is deterministic (\cite{gen}, \cite{Hoffmann1} and the review paper of Fan \cite{Fan2}), nonparametric estimation in the presence of the noise $(\epsilon_{j,n})$ substantially increases  the difficulty of the statistical problem. This is the topic of the present paper, and it can be related to practical issues in several application fields. In finance for instance, by considering the $Z_{j,n}$ as the result of a {\it latent} or unobservable efficient price $X_{i\Delta_n}$ corrupted by microstructure effects $\epsilon_{j,n}$ at scale $\Delta_n$, we obtain a more realistic model accounting for stylised facts on intraday scale usually attributed to bid-ask spread manipulation by market makers\footnote{This approach was grounded on empirical findings in the financial econometrics literature of the early years 2000 (among many others 
Ait-Sahalia {\it et al.} \cite{AMZ}, Mykland and Zhang \cite{MZ} and the references therein)}. Considering a diffusion perturbated by noise applies in other fields as well: in the context of functional MRI or fRMI, the problem of inference for diffusion processes with error measurement has been addressed by Donnet and Samson \cite{DS1, DS2} in an ergodic and parametric setting, when the sampling time $\Delta_n$ does not shrink to $0$ as $n \rightarrow \infty$. Se also Favetto and Samson \cite{FaS}. Recently, Schmisser \cite{Schmi1} has systematically studied the nonparametric estimation of the drift and the diffusion coefficient in an ergodic and mixed asymptotic setting, when $\Delta_n \rightarrow 0$ but $n\Delta_n \rightarrow \infty$. In this paper, we consider the nonergodic case, when only the diffusion coefficient can be identified, with $\Delta_n \rightarrow 0$ and $n\Delta_n$ fixed.

\subsection{Estimating the diffusion coefficient under additive noise: some history} \label{history}
\subsubsection*{Estimation of a finite-dimensional parameter and nonparametric functionals}
The first results about statistical inference of a diffusion with error measurement go back to Gloter and Jacod \cite{glo1, glo2} in 2001. They showed that if $\sigma_t = \sigma(t,\vartheta)$ is a deterministic function known up to a 1-dimensional parameter $\vartheta$, and if moreover the $\varepsilon_{j,n}$ are Gaussian and independent, then the LAN condition holds (Local Asymptotic Normality) for $\Delta_n = n^{-1}$  with rate $n^{-1/4}$. This implies that, even in the simplest Gaussian diffusion case, there is a substantial loss of information compared to the case without noise, where the standard $n^{-1/2}$ accuracy of estimation is achievable. 

At about the same time, the microstructure noise model for financial data was introduced by Ait-Sahalia, Mykland and Zhang in a series of papers \cite{AMZ, zha1, zha2}. Analogous approaches in various similar contexts progressively emerged in the financial econometrics literature: Podolskij and Vetter \cite{pod}, Bandi and Russell \cite{BR1, BR2}, Barndorff-Nielsen {\it et al.} \cite{barn} and the references therein. These studies tackled estimation problems in a sound mathematical framework, and incrementally gained in generality and elegance. A paradigmatic problem in this context is the estimation of the integrated volatility $\int_0^t \sigma_s^2ds$. Convergent estimators were first obtained by Ait-Sahalia {\it et al.} \cite{AMZ} with a suboptimal rate $n^{-1/6}$. Then the two-scale approach of Zhang \cite{zha2} achieved the rate $n^{-1/4}$. The Gloter-Jacod LAN property of \cite{glo1} for deterministic submodels shows that this cannot be improved. Further generalizations took the way of extending the nature of the latent price model $X$ (for instance \cite{AMZbis,seo, DS}) and the nature of the microstructure noise $(\epsilon_{j,n})$. It took some more time and contributions before Jacod and collaborators \cite{jac} took over the topic in 2007 with their simple and powerful pre-averaging technique, introduced earlier in a simplified context by Podolskij and Vetter \cite{pod}. 
In essence, it consists in first, smoothing the data as in signal denoising and then, apply a standard realised volatility estimator up to appropriate bias correction. Stable convergence in law is displayed for a wide class of pre-averaged estimators in a fairly general setting, closing somehow the issue of estimating the integrated volatility in a semiparametric setting.
\subsubsection*{Nonparametric inference}
In the nonparametric case, the problem is a little unclear. By nonparametric, one thinks of estimating the whole path $t \leadsto \sigma^2_t$.
However, since $\sigma^2 = (\sigma_t^2)_{t \geq 0}$ is usually itself genuinely random, there is no ``true parameter" to be estimated! When the diffusion coefficient is deterministic, the usual setting of statistical experiments is recovered. In that latter case, under the restriction that the microstructure noise process consists of i.i.d. noises, Munk and Schmidt-Hieber \cite{MSH, mun} proposed a Fourier estimator and showed its minimax rate optimality, extending a previous approach for the parametric setting (\cite{cai}). This approach relies on a formal analogy with inverse ill-posed problems. When the microstructure noises $(\epsilon_{j,n})$ are Gaussian i.i.d. with variance $\tau^2$, Rei\ss \,\cite{rei2} recently showed the asymptotic equivalence in the Le Cam sense with the observation of the random measure 
$$\sqrt{2\sigma}+\tau\, n^{-1/4}\dot B$$
where $\dot B$ is a Gaussian white noise. This is a beautiful and deep result: the normalisation $n^{-1/4}$ is illuminating when compared with the optimality results obtained by previous authors. 
\subsection{Our results}
The asymptotic equivalence proved in \cite{rei2} provides us with a benchmark for the complexity of the statistical problem and is inspiring: we target in this paper to put the problem of estimating nonparametrically the random parameter $t \leadsto \sigma^2_t$ to the level of classical denoising in the adaptive minimax theory.  In spirit, we follow the classical route of nonlinear estimation in de-noising, but we need to introduce new tools. Our procedure is twofold:
\begin{enumerate}
\item We approximate the random signal $t \leadsto \sigma_t^2$ by an atomic representation
\begin{equation} \label{representation}
\sigma^2_t \approx \sum_{\nu \in {\mathcal V}(\sigma^2)} \langle \sigma^2, \psi_\nu\rangle \psi_\nu(t)
\end{equation}
where $\langle \cdot,\cdot \rangle$ denotes the usual $L^2$-inner product and $\big(\psi_\nu, \nu \in {\mathcal V}(\sigma)\big)$ is a collection of wavelet functions that are localised in time and frequency, indexed by the set ${\mathcal V}(\sigma^2)$ that depends on the path $t \leadsto \sigma_t^2$ itself.
As for the precise meaning of the symbol $\approx$ and the property of the $\psi_\nu$'s, we do not specify yet.
\item We then estimate $\langle \sigma^2, \psi_\nu\rangle$ and specify a selection rule for ${\mathcal V}(\sigma)$ (with the dependence in $\sigma$ somehow replaced by an estimator). The rule is dictated by hard thresholding over the estimations of the coefficients $ \langle \sigma^2, \psi_\nu\rangle$ that are kept only if they exceed some noise level, tuned with the data, as in standard wavelet nonlinear approximation (Donoho, Johnstone, Kerkyacharian, Picard and collaborators \cite{don2, don, HKPT}).
\end{enumerate}
The key issue is therefore the estimation of the linear functionals
\begin{equation} \label{key issue}
\langle \sigma^2, \psi_\nu\rangle = \int_{\R} \psi_\nu(t)\sigma_t^2 dt.
\end{equation}
An important fact is that the functions $\psi_\nu$ are well located but oscillate, making the approximation of \eqref{key issue} delicate, in contrast to the global estimation of the integrated volatility: this is where we depart from the results of Jacod and collaborators \cite{jac,pod}. If we could observe the latent process $X$ itself at times $j\Delta_n$, then standard quadratic variation based estimators like
\begin{equation} \label{naive quadratic estimator}
\sum_{j} \psi_\nu (j\Delta_n)(X_{j\Delta_n}-X_{(j-1)\Delta_n})^2
\end{equation}
would give rate-optimal estimators of \eqref{key issue}, as follows from standard results on nonparametric estimation in diffusion processes \cite{gen, Hoffmann1, hof}. However, we only have a noisy version of $X$ via $(Z_{j,n})$ and further ``intermediate" de-noising is required.

At this stage, we consider local averages of the data $Z_{j,n}$ at an intermediate scale $m$ so that $\Delta_n \ll 1/m$ but $m \rightarrow \infty$. Let us denote loosely (and temporarily) by ${\text{Ave}}(Z)_{i,m}$ an averaging of the data $(Z_{j,n})$ around the point $i/m$. We have
\begin{equation} \label{loose approx}
{\text{Ave}}(Z)_{i,m} \approx X_{i/m} + \text{small noise}
\end{equation}
and thus we  have a de-blurred version of $X$, except that we must now handle the small noise term of \eqref{loose approx} and the loss of information due to the fact that we dispose of (approximate) $X_{i/m}$ on a coarser scale since $m \ll \Delta_n^{-1}$. We subsequently estimate \eqref{key issue} replacing the naive guess \eqref{naive quadratic estimator} by
\begin{equation} \label{the coeff estimator}
\sum_{i} \psi_\nu (i\Delta_n)\big[\big({\text{Ave}}(Z)_{i,m} - {\text{Ave}}(Z)_{i-1,m}\big)^2 + \text{bias correction}\big]
\end{equation}
up to a further bias correction term that comes from the fact that we take square approximation of $X$ via \eqref{loose approx}. In Section \ref{preliminaries estimator}, we generalise \eqref{the coeff estimator} to arbitrary kernels within a certain class of oscillating pre-averaging functions, in the same spirit as in Gloter and Hoffmann \cite{glo4} or Rosenbaum \cite{ros} where this technique is used for denoising stochastic volatility models corrupted by noise.

We prove in Theorems \ref{perfo est} and \ref{resultupper} an upper bound  for our procedure in $L^p$-loss error over a fixed time horizon. Assuming that the path $t \leadsto \sigma_t^2$ has $s$ derivatives in $L^\pi$ with a prescribed probability, the upper bound is of the form $n^{-\alpha/4}$ for an explicit $\alpha  = \alpha(s,p,\pi) < 1$ to within inessential logarithmic terms. We retrieve the expected results of wavelet thresholding over Besov spaces up to the noise rate $n^{-1/4}$ instead of the usual $n^{-1/2}$ in white Gaussian noise or density estimation, but that is inherent to the problem of microstructure noise, as already established in \cite{glo1}. It is noteworthy that, although the rates of convergence depend on the smoothness parameters $(s,\pi)$, the thresholding procedure does not, and is therefore adaptive in that sense. A major difficulty is that in order to employ the wavelet theory in this context, we must assess precise deviation bounds for quantities of the form \eqref{the coeff estimator}, which require delicate martingale techniques.
We prove in Theorem \ref{resultlower} that this result is sharp, even if $t \leadsto \sigma_t^2$ is random so that we do not have a statistical model in the strict sense. In order to encompass this level of generality, we propose a modification of the notion of upper and lower rate of estimation of a random parameter in Definition \ref{def upper rate} and \ref{lower rate}. This approach is presented in details in the methodology Section \ref{methodology}. 


The paper is organized as follows. In Section \ref{Main results} we introduce notation and formulate the key results. An explicit construction of the estimator can be found in Section \ref{construction}. 
Finally, the proofs of the main results and some (unavoidable) technicalities are deferred to Section \ref{the proofs}.

\section{Main results} \label{Main results}
\subsection{The data generating model}
We consider a continuous adapted 1-dimensional process $X$ of the form \eqref{defdiffusion}  on a filtered probability space $(\Omega, {\mathcal F}, ({\mathcal F}_t)_{t \geq 0}, \PP)$. Without loss of generality, we assume that $X_0=0$. 
\begin{assumption} \label{basic assumption}
The processes $\sigma$ and $b$ are c\`adl\`ag (right continuous with left limits), ${\mathcal F}_t$-adapted, and a weak solution of \eqref{defdiffusion} is unique and well defined. 

Moreover, a weak solution to $Y_t=\int_0^t \sigma_s dW_s$ is also unique and well defined, the laws of $X$ and $Y$ are equivalent on ${\mathcal F}_t$ and we have, for some $\rho >1$
\begin{equation} \label{drift elimination}
\mathbb{E}\big[\exp\big(\rho \int_0^t\frac{b_s}{\sigma_s^2}dY_s\big)\big]<\infty.
\end{equation}
\end{assumption}


We consider a fixed time horizon $T = n\Delta_n$, and with no loss of generality, we take $T = 1$ hence $\Delta_n = n^{-1}$. For $j = 0,\ldots, n$, we assume that we can observe a blurred version of $X$ at times $\Delta_n j = j/n$ over the time horizon $[0,T] = [0,1]$. The blurring accounts for microstructure noise at fine scales and takes the form
\begin{equation} \label{observationsbis}
Z_{j,n} := X_{j/n} +\epsilon_{j,n},\;\;j=0,1,\ldots,n
\end{equation}
where the microstructure noise process $(\epsilon_{j,n})$ is implicitly defined on the same probability space as $X$ and satisfies
\begin{assumption} \label{microstructure noise assumption}
We have
\begin{equation} \label{def micro noise}
\epsilon_{j,n} = a(j/n,X_{j/n})\eta_{j,n},
\end{equation}
where the function $(t,x)\leadsto a(t,x)$ is continuous and bounded. Moreover, the random variables $(\eta_{j,n})$ are independent, and independent of $X$. Moreover, for every $0 \leq j \leq n$ and $n\geq 1$, we have
$$\E\big[\eta_{j,n}\big]=0,\;\;\E\big[\eta_{j,n}^2\big]=1,\;\;\E\big[|\eta_{j,n}|^p\big]<\infty,\;p >0.$$
\end{assumption}
Given data $Z_\cdot=\{Z_{j,n},\,j=0,\ldots, n\}$ following \eqref{observations}, the goal is to estimate non-parametrically the random function $t \leadsto \sigma_t^2$ over the time interval $[0,1]$. Asymptotics are taken as the observation frequency $n \rightarrow \infty$.
\subsubsection*{Discussion on Assumptions \ref{basic assumption} and \ref{microstructure noise assumption}}
Assumption \ref{basic assumption} on $b$ and $\sigma$ is relatively weak, except for the moment condition \eqref{drift elimination}. This assumption is somewhat technical, for it enables to implicitly assume that $b=0$. Indeed, if $\PP_{\sigma,b}$ denotes the law of $(X_t)_{t \in [0,1]}$ with drift $b$ and volatility $\sigma$, we have by Girsanov's theorem 
\begin{align*}
  \frac{d\PP_{\sigma,b}}{d\PP_{\sigma,0}}= \exp\Big(\int_0^1\frac{b_s}{\sigma_s^2}dX_s
  - \frac 12 \int_0^1  \frac{b_s^2}{\sigma_s^2}ds\Big).
\end{align*}
By H\"older inequality, for a random variable $Z$, we derive
\begin{align}
  \E_{\sigma,b}\big[|Z|^p\big]^{1/p}
  &=
  \E_{\sigma,0}\Big[\tfrac{d\PP_{\sigma,b}}{d\PP_{\sigma,0}}|Z|^p\Big]^{1/p} \nonumber \\
  &\leq 
  \E_{\sigma,0}\big[\exp\big(\rho\int_0^1\frac{b_s}{\sigma_s^2}dX_s
  \big)\big]^{1/(p\rho)}
  \E_{\sigma,0}\big[|Z|^{\overline p} \ \big]^{1/\overline p} \label{trick holder girsanov}
\end{align}
with $\overline p= p\rho/(\rho-1).$ Therefore, Condition \eqref{drift elimination} guarantees that if we have an estimate of the form $\E_{\sigma,_0}\big[|Z|^p\big]^{1/p} \leq c_p\, n^{-\gamma}$ for any $p\geq 1$ and for some $\gamma>0$, then the same property holds replacing $\PP_{\sigma,0}$ by $\PP_{\sigma,b}$, up to a modification of the constant $c_p$. Thus Condition \eqref{drift elimination} is a useful tool that enables to condense the proofs in many places afterwards. It is satisfied as soon as $\sigma$ is bounded below and $b$ has appropriate integrability conditions. In some cases of interest where it may fail to hold, one can still proceed by working directly under $\PP_{\sigma,b}$. 

Concerning Assumption \ref{microstructure noise assumption}, we assume a relatively weak scheme of microstructure noise, by assuming that the 
$\epsilon_{j,n}$ form a martingale array that may depend on the unobserved process $X$ through a function $t \leadsto a(t,X_t)$ as the standard deviation of the additive noise. This enables richer structures than simple additive independent noise. One may wish to relax further Assumption \ref{microstructure noise assumption} by assuming a correlation decay only, but again, for technical reason, we keep to this simpler framework. 

\subsection{Statistical methodology} \label{methodology}
\subsubsection*{Recovering $\sigma^2$ over a function class ${\mathcal D}$}
Strictly speaking, since the target parameter $\sigma^2  = (\sigma_t^2)_{t\in [0,1]}$ is random itself (as an ${\mathcal F}$-adapted process), we cannot assess the performance of an ``estimator of $\sigma^2$" in the usual way.  We need to modify the usual notion of convergence rate over a function class.  


\begin{definition} \label{def upper rate}
An estimator of $\sigma^2 = (\sigma_t^2)_{t\in [0,1]}$ is a random function 
$$t\leadsto \widehat \sigma_n^2(t),\;\;t \in [0,1],$$ 
measurable with respect to the observation $(Z_{j,n})$ defined in \eqref{observations}.
\end{definition}
We need to modify the usual notion of convergence rate. Let us denote by ${\mathcal D}$ a class of real-valued functions defined on $[0,1]$. 
\begin{definition} \label{defupperbound}
We say that the rate $0< v_n \rightarrow 0$ (as $n\rightarrow \infty$) is achievable for estimating $\sigma^2$ in $L^p$-norm over ${\mathcal D}$ if there exists an estimator $\widehat \sigma_n^2$ such that
\begin{equation} \label{upperbound}
\limsup_{n \rightarrow \infty}v_n^{-1}\E\Big[\|\widehat \sigma_n^2-\sigma^2\|_{L^p([0,1])}\mathbb{I}_{\big\{\sigma^2\in {\mathcal D}\big\}}\Big]<\infty.
\end{equation}
\end{definition}
\begin{remark}
If we wish $(\sigma_t)$ to be deterministic, we can make {\it a priori} assumptions so that the condition $\sigma^2\in {\mathcal D}$ is satisfied, in which case we simply ignore the indicator in \eqref{upperbound}. In other cases, this condition will be satisfied with some probability (see below). But it may also well happen that for some choices of ${\mathcal D}$ we have $\PP\big[\sigma^2 \in {\mathcal D}\big]=0$ in which case the upper bound \eqref{upperbound} becomes trivial and noninformative.
\end{remark}
In this context, a sound notion of optimality is unclear. We propose the following
\begin{definition} \label{lower rate}
The rate $v_n$ is a lower rate of convergence over ${\mathcal D}$ in $L^p$ norm if there exists a filtered probability space $(\widetilde \Omega, \widetilde {\mathcal F}, (\widetilde {\mathcal F}_t)_{t \geq 0}, \widetilde \PP)$, a process $\widetilde X$ defined on $(\widetilde \Omega, \widetilde {\mathcal F})$ with the same distribution as $X$ under Assumptions \ref{basic assumption} together with a process 
$(\widetilde \epsilon_{j,n})$ satisfying \eqref{def micro noise} with $\widetilde X$ in place of $X$, such that Assumption  \ref{microstructure noise assumption} holds, and moreover:
\begin{equation} \label{first cond lb}
\widetilde \PP\big[\sigma^2\in {\mathcal D}\big] >0
\end{equation}
and
\begin{equation} \label{lower bound}
\liminf_{n \rightarrow \infty}v_n^{-1}\inf_{\widehat \sigma_n^2}\widetilde \E\Big[\|\widehat \sigma_n^2-\sigma^2\|_{L^p([0,1])}\mathbb{I}_{\big\{\sigma^2\in {\mathcal D}\big\}}\Big]>0,
\end{equation}
where the infimum is taken over all estimators.
\end{definition}
Let us elaborate on Definition \ref{lower rate}: as already mentioned, $\sigma^2$ is ``genuinely" random, and we cannot say that our data $\{Z_{j,n}\}$ generate a statistical experiment as a family of probability measures indexed by some parameter of interest. Rather, we have a fixed probability measure $\PP$, but this measure is only ``loosely" specified by very weak conditions, namely Assumptions  \ref{basic assumption} and \ref{microstructure noise assumption}. A lower bound as in Definition \ref{lower rate} says that, given a model $\PP$, there exists a probability measure $\widetilde \PP$, possibly defined on another space so that Assumptions \ref{basic assumption} and \ref{microstructure noise assumption} hold under $\widetilde \PP$ together with \eqref{lower bound}.  Without further specification on our model, there is no sensible way to discriminate between $\PP$ and $\widetilde \PP$ since both measures (and the accompanying processes) satisfy Assumptions \ref{basic assumption} and \ref{microstructure noise assumption}; moreover, under $\widetilde \PP$, we have a lower bound.

\subsubsection*{Function classes: wavelets and Besov spaces} \label{WavBas}

We describe the smoothness of a function by means of Besov spaces on the interval. A thorough account of Besov spaces  ${\mathcal B}^s_{\pi,\infty}$ and their connection to wavelet bases in a statistical setting are discussed in details in the classical papers of Donoho {\it et al.} \cite{don} and Kerkyacharian and Picard \cite{ker}. Let us recall some fairly classical\footnote{We follow closely the notation of Cohen \cite{coh}.} material about Besov spaces through their characterisation in terms of wavelets. We use $n_0$-regular wavelet bases $(\psi_\nu)_\nu$ adapted to the domain $[0,1]$. More precisely, the multi-index $\nu$ concatenates the spatial index and the
resolution level $j=|\nu|$. We set
$\Lambda_j:=\{\nu,\;|\nu|=j\}$ and $\Lambda:=\cup_{j \geq
-1} \Lambda_j$. Thus for $f \in L^2([0,1])$, we have
$$f = \sum_{j \geq -1} \sum_{\nu \in \Lambda_j}
\langle f,\psi_\nu\rangle \psi_\nu=\sum_{\nu \in \Lambda} \langle f,\psi_\nu\rangle \psi_\nu,$$
where we have set $j:=-1$ in order to incorporate the low frequency
part of the decomposition. From now on the basis
$(\psi_{\nu})_\nu$ is fixed and depends on a regularity index $n_0$ which role is specified in Assumption \ref{AssumptionOnBasis} below. 
\begin{definition}For $s>0$ and $\pi\in (0,\infty]$, a function $f:[0,1]\rightarrow \R$ belongs to the Besov space ${\mathcal B}_{\pi,\infty}^s([0,1])$ if the following norm is finite:
\begin{equation} \label{besovanalysis}
\|f\|_{{\mathcal B}^s_{\pi,\infty}([0,1])}:= \sup_{j \geq
-1}2^{j\big(s+\tfrac{1}{2}-\tfrac{1}{\pi}\big)}\big(\sum_{\nu\in
\Lambda_j}|\langle f, \psi_\nu\rangle|^\pi\big)^{1/\pi},
\end{equation}
with the usual modification if $\pi=\infty$.
\end{definition}
Precise connection between this definition of Besov norm and more
standard ones can be found in \cite{coh, coh2}. Given a basis $(\psi_\nu)_\nu$ with regularity index $n_0>0,$ the Besov space defined by \eqref{besovanalysis} exactly matches the
usual definition in terms of modulus of smoothness for $f,$ provided that $\pi\ge1$ and $s \leq n_0.$ A particular case include the H\"older space ${\mathcal C}^s([0,1])={\mathcal B}^s_{\infty,\infty}([0,1])$. Moreover, the following Sobolev embedding 
inequality holds 
$$
 \norme{f}_{{\mathcal B}^{s_2}_{\pi_2,\infty}([0,1])} \le
\norme{f}_{{\mathcal B}^{s_1}_{\pi_1,\infty}([0,1])}
 \text{ for }
 s_1-1/\pi_1=s_2-1/\pi_2,\; \pi_2 \ge \pi_1,
$$
showing in particular that ${\mathcal B}^{s}_{\pi,\infty}([0,1])$ is embedded into continuous functions as soon as $s>1/\pi$.
The additional properties of the wavelet basis
$(\psi_\nu)_\nu$ that we need are summarized in the next
assumption. 
\begin{assumption}[Properties of the basis $(\psi_\nu)_\nu$]  \label{AssumptionOnBasis}
For $\pi \geq 1:$
\begin{itemize}
\item We have $\|\psi_\nu\|_{L^\pi([0,1])}^\pi
\sim 2^{|\nu|(\pi/2-1)}$.
\item For some arbitrary $n_0 >0$ and for all $s \le n_0$, $j_0  \ge 0$, we have
\begin{equation}\label{E:std_approx_besov}
\|f-\sum_{j \le j_0} \sum_{\nu\in \Lambda_j}
f_{\nu}\psi_{\nu}\|_{L^\pi([0,1])} \lesssim 2^{-{j_0} s}
\norme{f}_{{\mathcal B}_{\pi,\infty}^s([0,1])}.
\end{equation}
\item For any  $\Lambda_0 \subset \Lambda$,
\begin{equation} \label{superconcentration}
\int_{[0,1]}\Big(\sum_{\nu \in \Lambda_0}|\psi_{\nu}(x)|^2\Big)^{\pi/2}dx \sim \sum_{\nu \in \Lambda_0} \|\psi_\nu\|_{L^\pi([0,1])}^\pi.
\end{equation}
\item If $\pi > 1$, for any sequence $(u_\nu)_{\nu \in \Lambda}$
\begin{equation} \label{unconditional}
 \big\|\big(\sum_{\nu \in \Lambda}|u_\nu \psi_{\nu}|^2\big)^{1/2}\big\|_{L^\pi([0,1])} \sim \|\sum_{\nu \in \Lambda} u_\nu \psi_\nu\|_{L^\pi([0,1])}.
\end{equation}
\end{itemize}
\end{assumption}
\noindent The symbol $\sim$ means inequality in both ways, up to a
constant depending on $\pi$ only. The property
\eqref{E:std_approx_besov} reflects that our definition
\eqref{besovanalysis} of Besov spaces matches the definition in term
of linear approximation. Property \eqref{unconditional} means an
unconditional basis property and \eqref{superconcentration} is referred to as a superconcentration
inequality see \cite{ker}. The existence of compactly supported wavelet bases satisfying Assumption \ref{AssumptionOnBasis} goes back to Daubechies and is discussed for instance in
\cite{coh}.

We are interested in the case where $\sigma^2$ may belong to various smoothness classes, that include the case where $\sigma^2$ is deterministic and has as many derivatives as one wishes, but also the case of genuinely random processes that oscillate like diffusions, or fractional diffusions and so on. These smoothness properties are usually modelled in terms of Besov balls 
\begin{equation} \label{defBesov}
{\mathcal B}^s_{\pi,\infty}(c):=\big\{f:[0,1]\rightarrow \R,\;\|f\|_{{\mathcal B}^s_{\pi,\infty}([0,1])}\leq c\big\},\;\;c>0.
\end{equation}
that measure smoothness of degree $s>1/\pi$ in $L^{\pi}$ over the interval $[0,1]$, for $\pi \in (0,\infty)$. The restriction $s>1/\pi$ ensures that the functions in $\mathcal{B}^s_{\pi,\infty}$ are continuously embedded into H\"older continuous functions with index $s-1/\pi$. Besov balls also give a flexible way to describe the smoothness of the path of a continuous random process. For instance, if $(\sigma_t)$ is an It\^o continuous semimartingale itself with regular coefficients, we have $$\PP\big[\sigma^2 \in {\mathcal B}^{1/2}_{\pi,\infty}(c)\big] >0,\;\;\;\text{for every}\;\;\pi >1/2,$$
If it is a smooth transformation of a fractional Brownian motion with Hurst index, $H$, we have $\PP\big[\sigma^2 \in {\mathcal B}^{H}_{\pi,\infty}(c)\big] >0$ for $\pi > H$ likewise. The proof of such classical results can be found in Ciesielski {\it et al.} \cite{CiesielskyKerkyacharianRoynette}.


\subsection{Achievable estimation error bounds}
For prescribed smoothness classes of the form  ${\mathcal D} = {\mathcal B}^s_{\pi,\infty}(c)$ and $L^p$-loss functions, the rate of convergence $v_n$ depends on the index $s,\pi$ and $p$. Define the rate exponent
\begin{equation} \label{def exponent}
\alpha(s,p,\pi) = \min\Big\{\frac{s}{2s+1},\frac{s+1/p-1/\pi}{1+2s-2/\pi}\Big\}.
\end{equation}

\begin{theorem} \label{resultupper}
Work under Assumptions \ref{basic assumption} and \ref{microstructure noise assumption}. Then, for every $c>0$,  the rate $n^{-\alpha(s,p,\pi)/2}$
is achievable over the class ${\mathcal B}^s_{\pi,\infty}(c)$ in $L^p$-norm with $p \in [1,\infty)$, provided $s>1/\pi$ and $\pi \in (0,\infty)$, up to logarithmic corrections. 

Moreover, under Assumption \ref{AssumptionOnBasis}, the estimator explicitly constructed in Section \ref{final estimator} below attains this bound in the sense of \eqref{upperbound}, up to logarithmic corrections.
\end{theorem}

\begin{remark}
A (technical) restriction is that we assume $s >1/\pi$, a condition that guarantees  
some minimal H\" older smoothness for the path of $t \leadsto \sigma^2_t$.
\end{remark} 

\begin{remark}
The parametric rate $n^{-1/2}$ (formally obtained when letting $s\rightarrow \infty$ in the definition of $\alpha(s,p,\pi)$) has to be replaced by $n^{-1/4}$. This effect is due to microstructure noise, and was already identified in earlier parametric models as in Gloter and Jacod \cite{glo1} and subsequent works, both in parametric, semiparametric and nonparametric estimation, as follows from \cite{glo1, glo2, cai, MSH, zha2, jac} among others. 
\end{remark}

Our next result shows that this rate is nearly optimal in many cases.

\begin{theorem} \label{resultlower}
In the same setting as in Theorem \ref{resultupper}, assume moreover that $s-1/\pi>\tfrac{1+\sqrt{5}}{4}$. Then the rate $n^{-\alpha(s,p,\pi)/2}$
is a lower rate of convergence over ${\mathcal B}^s_{\pi,\infty}(c)$ in $L^p$ in the sense of Definition \ref{lower rate}.
\end{theorem}

Since the upper and lower bound agree up to some (inessential) logarithmic corrections, our result is nearly optimal in the sense of Definitions \ref{defupperbound} and \ref{lower rate}.

The proof of the lower bound is an application of a recent result of Rei\ss \;\cite{rei2} about asymptotic equivalence between the statistical model obtained by letting $\sigma^2$ be deterministic and the microstructure noise white Gaussian with an appropriate infinite dimensional Gaussian shift experiment. In particular, the restriction $s-1/\pi>\tfrac{1+\sqrt{5}}{4}$ stems from the result of Rei\ss \;and could presumably be improved. Our proof relies on the following strategy: we transfer the lower bound into a Bayesian estimation problem by constructing $\widetilde \PP$ adequately. We then use the asymptotic equivalence result of Rei\ss \;in order to approximate the conditional law of the data given $\sigma$ under $\widetilde \PP$ by a classical Gaussian shift experiment, thanks to a Markov kernel. In the special case $p=\pi=2$, we could also derive the result by using the lower bound in \cite{MSH}. Also, this setting may also enable to retrieve the standard minimax framework when $\sigma^2$ is deterministic and belongs to a Besov ball ${\mathcal B}^s_{\pi,\infty}(c)$. In that case, it suffices to construct a probability measure $\widetilde \PP$ such that under $\widetilde \PP$, the random variable $\sigma^2$ has distribution $\mu(d\sigma^2)$ with support in ${\mathcal B}^s_{\pi,\infty}(c)$, and is chosen to be a least favourable prior as in standard lower bound nonparametric techniques. It remains to check that Assumptions \ref{basic assumption} and \ref{microstructure noise assumption} are satisfied $\mu$-almost surely. We elaborate on this approach in the proof of Theorem \ref{resultlower} below.

\section{Wavelet estimation and pre-averaging} \label{construction}
\subsection{Estimating linear functionals}  \label{preliminaries estimator}
We estimate $\sigma^2$ via linear functionals of the form
$$\langle \sigma^2, h_{\ell k}\rangle := \int_0^1 2^{\ell/2}h(2^\ell t-k)d\langle X\rangle_t.$$
With no possible confusion, we denote by $\langle\cdot,\cdot\rangle$ the inner product of $L^2([0,1])$ and by
$$\langle X\rangle_t = (\PP-\lim)_{\delta \rightarrow 0}\sum_{t_i, \ t_i-t_{i-1}\leq \delta}(X_{t_i}-X_{t_{i-1}})^2$$
the quadratic variation of the continuous semimartingale $X$.
Here, the integers $\ell \geq 0$ and $k$ are respectively a resolution level and a location parameter. The test function $h:\R \rightarrow \R$ is smooth and throughout the paper we will assume that $h$ is compactly supported on $[0,1].$ Thus, $h_{\ell k}= 2^{\ell/2}h(2^\ell \cdot-k)$ is essentially located around $(k+\tfrac 12 )/2^\ell.$
\begin{definition} \label{preaverage function}
We say that $\lambda : [0,2)\rightarrow \R$ is a pre-averaging function if it is piecewise Lipschitz continuous, satisfies $\lambda(t)=-\lambda(2-t),$ and is not zero identically. 
To each pre-averaging function $\lambda$ we associate the quantity
$$\overline{\lambda}:=\Big(2\int_{0}^1 \big(\int_0^s \lambda(u)du\big)^2 ds\Big)^{1/2}$$
and define the (normalized) pre-averaging function $\widetilde \lambda:=\lambda/\overline{\lambda}$.
\end{definition}
For $1 \leq m < n$ and a sequence $(Y_{j,n}, j=0,\ldots, n)$, we define the pre-averaging of $Y$ at scale $m$ relative to $\lambda$ by setting for $i=2,\ldots,m$
\begin{equation} \label{pre averaging}
\overline{Y}_{i,m}(\lambda):=\frac{m}{n}\sum_{\tfrac{j}{n}\in \big(\tfrac{i-2}{m},\tfrac{i}{m}\big]}\widetilde \lambda\big(m \tfrac{j}{n} -(i-2)\big)Y_{j,n},
\end{equation}
the summation being taken w.r.t. the index $j$.
If $Y_{j,m}$ has the form $Y_{j/m}$ for some underlying continuous time process $t \leadsto Y_t$, the pre-averaging of $Y$ at scale $m$ is a kind of local average that mimics the behaviour of $Y_{i/m}-Y_{(i-2)/m}$. Indeed, using $\lambda(t)=-\lambda(2-t),$ for $t\in (0,1],$
\begin{equation*}
 \overline{Y}_{i,m}(\lambda)\approx-\frac{m}{n}\sum_{\tfrac{j}{n}\in \big(0,\tfrac{1}{m}\big]}\widetilde \lambda\big(m \tfrac{j}{n}\big)(Y_{i/m-j/n}-Y_{(i-2)/m+j/n}).
\end{equation*}
Thus, $\overline{Y}_{i,m}(\lambda)$ might be interpreted as a sum of differences in the interval $[(i-2)/m, i/m]$, weighted by $\widetilde \lambda.$ 

From \eqref{naive quadratic estimator}, a first guess for estimating $\langle \sigma^2, h_{\ell k}\rangle$ is to consider the quantity
$$\sum_{i=2}^m h_{\ell k}\big(\tfrac{i-1}{m}\big)\overline{Z}_{i,m}^2$$
for some intermediate scale $m$ that needs to be tuned with $n$ and that reduces the effect of the noise $(\varepsilon_{j,n})$ in the representation \eqref{observations}. However, such a procedure is biased and a further correction is needed. To that end, we introduce
\begin{align} \label{bias correction}
& \mathfrak{b}(\lambda, Z_\cdot)_{i,m} 
:=  \frac{m^2}{2n^2}\sum_{\tfrac{j}{n}\in \big( \tfrac{i-2}{m}, \tfrac{i}{m}\big]}\widetilde \lambda^2\big(m\tfrac{j}{n}-(i-2)\big)\big(Z_{j,n}-Z_{j-1,n}\big)^2
\end{align}
In order to get a first intuition, note that $(Z_{j,n}-Z_{j-1,n})^2 \approx (\epsilon_{j,n}-\epsilon_{j-1,n})^2.$ Further stochastic approximations, detailed in the proof in Section \ref{proof of proposition debut}, show that subtracting $\mathfrak{b}(\lambda, Z_\cdot)_{i,m}$ corrects in a natural way for the bias induced by the additive microstructure noise.

 
Finally, our estimator of $\langle \sigma^2, h_{\ell k}\rangle$ is
\begin{equation} \label{def est fonct lin}
{\mathcal E}_m(h_{\ell k}):=\sum_{i=2}^m h_{\ell k}\big(\tfrac{i-1}{m}\big)\big[ \hspace{1pt} \overline{Z}_{i,m}^2-\mathfrak{b}(\lambda,Z_\cdot)_{i,m}\big].
\end{equation}

\subsection{The wavelet threshold estimator} \label{wavelet estimator}
Let $(\varphi,\psi)$ denote a pair of scaling function and mother wavelet that generate a wavelet basis $(\psi_\nu)_\nu$ satisfying Assumption \ref{AssumptionOnBasis}. The random function $t \leadsto \sigma_t^2\;$ taken path-by-path as an element of $L^2([0,1])$ has for every non-negative integer $\ell_0$ an almost-sure representation
\begin{equation} \label{expansion volatility}
\sigma_\cdot^2 = \sum_{k\in \Lambda_{{\ell_0}}}c_{\ell_0k}\,\varphi_{\ell_0k}(\cdot)+\sum_{\ell > \ell_0}\sum_{k \in \Lambda_\ell} d_{\ell k}\,\psi_{\ell k}(\cdot),
\end{equation}
with 
$c_{\ell_0k} = \langle \sigma^2,\varphi_{\ell _0k}\rangle=\int_0^1 \varphi_{\ell _0k}(t)d\langle X\rangle_t$
and 
$d_{\ell k}  = \langle \sigma^2,\psi_{\ell k}\rangle= \int_0^1 \psi_{\ell k}(t)d\langle X\rangle_t.$
For every $\ell \geq 0$, the index set $\Lambda_\ell $ has cardinality $2^\ell $ (and also incorporates boundary terms in the first part of the expansion that we choose not to distinguish in the notation from $\varphi_{\ell _0k}$ for simplicity.) The choice of $\ell_0$ in (\ref{expansion volatility}) determines the representation of $\sigma^2$ as sum of a low resolution approximation based on the scaling function $\phi$ and a high-frequency wavelet decomposition, Section \ref{methodology}.
Following the standard wavelet threshold algorithm (see for instance \cite{don} and in its more condensed form \cite{ker}), we approximate Formula \eqref{expansion volatility} by
\begin{equation} \label{def est}
\widehat \sigma^2_n(\cdot):=\sum_{k\in \Lambda_{{\ell_0}}}{\mathcal E}(\varphi_{\ell_0k})\varphi_{\ell_0k}(\cdot)+\sum_{\ell = \ell_0+1}^{\ell_1}\sum_{k \in \Lambda_\ell} {\mathcal T}_\tau\big[{\mathcal E}(\psi_{\ell k})\big]\psi_{\ell k}(\cdot)
\end{equation}
where the wavelet coefficient estimates ${\mathcal E}(\varphi_{\ell _0k})$ and ${\mathcal E}(\psi_{\ell k})$ are given by \eqref{def est fonct lin} and
$${\mathcal T}_\tau[x]=x1_{\{|x|\geq \tau\}},\;\;\tau \geq 0,\;\;x\in \R$$
is the standard hard-threshold operator. Thus $t \leadsto \widehat \sigma_n^2(t)$ is specified by the resolution levels $\ell_0$, $\ell_1$, the threshold $\tau$ and the estimators ${\mathcal E}(\varphi_{\ell_0k})$ and ${\mathcal E}(\psi_{\ell k})$ which in turn are entirely determined by the choice of the pre-averaging function $\lambda$ and the pre-averaging resolution level $m$. (And of course, the choice of the basis generated by $(\varphi,\psi)$ on $L^2([0,1])$.)

\subsection{Convergence rates} \label{final estimator}

We first give two  results on the properties of ${\mathcal E}_m(h_{\ell k})$ for estimating $\langle \sigma^2, h_{\ell k}\rangle_{L^2}$.
\begin{theorem}[Moment bounds] \label{moment bounds}
Work under Assumptions \ref{basic assumption} and \ref{microstructure noise assumption}. Let us assume that $h$ admits a piecewise Lipschitz derivative and that $2^\ell \leq m\leq n^{1/2}$. 

If $s>1/\pi$, for any $c>0$, for every $p \geq 1$, we have
\begin{align*}
\E\big[\big|{\mathcal E}_m(h_{\ell k})-\langle \sigma^2, h_{\ell k}\rangle\big|^p\mathbb{I}_{\{\sigma^2 \in {\mathcal B}^s_{\pi,\infty}(c)\}}\big] 
\lesssim & \;m^{-p/2}\\
&+m^{-\min\{s-1/\pi,1\}p}|h_{\ell k}|_{1,m}^p,
\end{align*}
where 
$|h_{\ell k}|_{1,m}:=m^{-1}\sum_{i = 1}^m |h_{\ell k}(i/m)|.$
The symbol $\lesssim$ means up to a constant that does not depend on $m$ and $n$.
\end{theorem}

\begin{theorem}[Deviation bounds] \label{deviation bounds}
Work under Assumptions \ref{basic assumption} and \ref{microstructure noise assumption}. Let us assume that $h$ admits a piecewise Lipschitz derivative and that $2^{\ell} \leq m \leq n^{1/2}$. If moreover
$$m2^{-\ell} \geq m^{q},\;\;\text{for some}\;\;q>0,$$
then, if $s>1/\pi$, for any $c>0$, for every $p\geq 1$, we have
$$\PP\Big[\big|{\mathcal E}_m(h_{\ell k})-\langle \sigma^2, h_{\ell k}\rangle\big|\geq \kappa \big(\tfrac{p \log m}{m}\big)^{1/2}\;,\;\sigma^2 \in {\mathcal B}^s_{\pi,\infty}(c)\Big] \lesssim m^{-p}$$
provided
$$\kappa > 4 \big(\tfrac{\rho}{\rho-1}\big)^{1/2}\Big(\overline c+\sqrt{2 \hspace{2pt} \overline c} \ \|a\|_{L^\infty}\|\lambda\|_{L^2}\overline{\lambda}^{ \hspace{2pt} -1}+\|a\|_{L^\infty}^2\|\lambda\|_{L^2}^2\overline{\lambda}^{ \hspace{2pt} -2}\Big)$$
and
$$m^{-(s-1/\pi)}|h_{\ell k}|_{1,m} \lesssim m^{-1/2},$$
where $\overline c:= \sup_{\sigma^2 \in {\mathcal B}^s_{\pi,\infty}(c)}\|\sigma^2\|_{L^\infty}.$
\end{theorem}
\begin{theorem}\label{perfo est}
Work under Assumptions \ref{basic assumption}, \ref{microstructure noise assumption} and \ref{AssumptionOnBasis}. 
Let $\widehat \sigma_n^2$ denote the wavelet estimator defined in \eqref{def est}, constructed from $(\varphi,\psi)$ and a pre-averaging function $\lambda$,  such that
$$
m  \sim n^{1/2},
2^{\ell_0} \sim m^{1-2\alpha_0}\;\;\text{for some}\;\;0 < \alpha_0 < 1/2, \; 
2^{\ell_1}  \sim m^{1/(1+2\alpha_0)}$$
and $\tau  := \widetilde \kappa \sqrt{\tfrac{\log m}{m}}$ for sufficiently large $\widetilde \kappa >0$.
Then, for 
$$\alpha _0+1/\pi \leq s \leq \max\{\alpha_0/(1-2\alpha_0),n_0\},$$
the estimator $\widehat \sigma_n^2$ achieves \eqref{upperbound} over ${\mathcal D} = {\mathcal B}^s_{\pi,\infty}(c)$ with $v_n = n^{-\alpha(s,p,\pi)/2}$ up to logarithmic factors. As a consequence, we have Theorem \ref{resultupper}.
\end{theorem}
\begin{proof}
Thanks to Theorems \ref{moment bounds} and \ref{deviation bounds}, Theorem \ref{perfo est} is now a consequence of the general theory of wavelet threshold estimators, as developed by Kerkyacharian and Picard \cite{ker}. To that end, it suffices to obtain appropriate moment bounds and large deviation inequalities for estimators of wavelet coefficients in wavelet bases satisfying Assumption \ref{AssumptionOnBasis}. 

More precisely, by assumption, we have $ s-1/\pi \geq \alpha_0$ and $2^{\ell_0} \sim m^{1-2\alpha_0}$ therefore, the term 
$m^{-\min\{s-1/\pi,1\}}|h_{\ell k}|_{1,m}$ is less than a constant times 
$$m^{-\alpha_0} 2^{-\ell/2} \lesssim  m^{-\alpha_0} m^{-(1-2\alpha_0)/2}\sim m^{-1/2},$$
where we used that $|h_{\ell k}|_{1,m} \lesssim 2^{-\ell/2}$ with $h=\varphi$.
This together with Theorem \ref{moment bounds} shows that we have the moment bound
$$
\E\big[\big|{\mathcal E}_m(\varphi_{\ell_0 k})-\langle \sigma^2, \varphi_{\ell_0 k}\rangle\big|^p\mathbb{I}_{\{\sigma^2 \in {\mathcal B}^s_{\pi,\infty}(c)\}}\big] 
\lesssim m^{-p/2} \lesssim n^{-p/4},\\
$$
so that Condition (5.1) of Theorem 5.1 in Kerkyacharian and Picard \cite{ker} is satisfied with $c(n)=(\log n/n)^{1/4}$ and $\Lambda(n)=n^{1/2}$ with the notation of \cite{ker}. In the same way, by Theorem \ref{deviation bounds}, with $h=\psi$, for every $p \geq 1$, we obtain, for a large enough $\kappa$ the deviation bound
$$
\PP\Big[\big|{\mathcal E}_m(\psi_{\ell k})-\langle \sigma^2, \psi_{\ell k}\rangle\big|\geq \kappa\big(\tfrac{p\log m}{m}\big)^{1/2}\;,\;\sigma^2 \in {\mathcal B}^s_{\pi,\infty}(c)\Big] \lesssim m^{-p} \lesssim n^{-p/2}
$$
and therefore Condition (5.2) of Theorem 5.1 in \cite{ker} is satisfied with the same specification. This is all that is required to apply the wavelet threshold algorithm: by Corollary 5.2 and Theorem 6.1 of \cite{ker} we obtain \eqref{upperbound} hence Theorem \ref{resultupper}.
\end{proof}

\begin{remark}
By taking $\alpha_0 <1/2$, Theorem \ref{perfo est} shows that in this case the estimator can at most adapt to the correct smoothness within the range $\alpha_0+1/\pi \leq s \leq \alpha_0/(1-2\alpha_0)<\infty.$   
\end{remark}
\section{Proofs} \label{the proofs} 
\subsection{Proof of Theorem \ref{moment bounds}} \label{proof of proposition debut}
We shall first introduce several auxiliary estimates which rely on classical techniques of discretization of random processes. 
Unless otherwise specified, $L^2$ abbreviates $L^2([0,1])$ and likewise for $L^\infty$.

If $g:[0,1]\rightarrow \R$ is piecewise continuously differentiable, we define for $n \geq 1$
\begin{align}
 \mathfrak{R}_n(g):=\Big(\sum_{j=1}^n\int_{(j-1)/n}^{j/n}\big(\tfrac{1}{n}\sum_{l=j}^ng'(\tfrac{l}{n})-\int_s^1g'(u)du\big)^2ds\Big)^{1/2},
  \label{eq.Rdef}
\end{align}
and
$$|g|_{p,m}:=\Big(\tfrac 1m \sum_{i=1}^m |g(\tfrac{i-1}m)|^p\Big)^{1/p}.$$
In the following, if ${\mathcal D}$ is a function class, we will sometimes write
$\E_{{\mathcal D}}[\cdot]$ for $\E[\cdot \;\mathbb{I}_{\sigma^2 \in {\mathcal D}}]$. Clearly, if ${\mathcal D_1}\subset {\mathcal D}_2$, we have for non-negative integrands $\E_{{\mathcal D}_1}[\cdot]\leq \E_{{\mathcal D}_2}[\cdot]$. For $c>0$, let
$${\mathcal D}_\infty(c):=\{f:[0,1]\rightarrow \R,\;\|f\|_{L^\infty}\leq c\}.$$
Throughout the remaining part of this paper, we extend pre-averaging functions to the real line by $\lambda(t)=0$ for all $t\in \mathbb{R}\setminus[0,2).$

\subsubsection*{Preliminaries : some estimates for the latent price $X$}
\begin{lemma}[Discretisation effect] \label{lemma discretization}
Let $g: [0,1]\rightarrow \mathbb{R},$ be a deterministic function with piecewise continuous derivative, such that $g(1)=0.$
Work under Assumption \ref{basic assumption}.
For every $p \geq 1$ and $c>0$, we have
\begin{align*}
\E_{{\mathcal D}_\infty(c)}\Big[\Big|\big(\tfrac{1}{n}\sum_{i = 1}^n g'(\tfrac{i}{n})X_{i/n}\big)^2
-\big(\int_0^1g(s)dX_s\big)^2\Big|^p\Big] 
\lesssim \|g\|_{L^2}^{p}\,\mathfrak{R}_n^{p}(g)+\mathfrak{R}^{2p}_n(g).
\end{align*}
\end{lemma}
\begin{proof}
By Assumption \ref{basic assumption}, using \eqref{trick holder girsanov} and anticipating that rates of convergence are in power of $n$, we may (and will) assume that $X$ is a local martingale and take subsequently $b=0$. Next, by Cauchy-Schwarz, we split the error term into a constant times $I \times II + III \times II$, with
\begin{align*}
I & := \E_{{\mathcal D}_\infty(c)}\Big[\Big|
	\int_0^1 g(s) dX_s \Big|^{2p}\Big]^{1/2}, \\
II & :=\E_{{\mathcal D}_\infty(c)}\Big[\Big|
	\tfrac 1n \sum_{j=1}^n g'\big(\tfrac jn\big) X_{j/n}+\int_0^1 g(s) dX_s\Big|^{2p}\Big]^{1/2},
	 \\
III &:=  \E_{{\mathcal D}_\infty(c)}\Big[\Big|
	\tfrac 1n \sum_{j=1}^n g'\big(\tfrac jn\big) X_{j/n}\Big|^{2p}\Big]^{1/2}\lesssim I+II.
\end{align*}
Define the stopping time
$$T_c:=\inf\{s \geq 0,\;\;\sigma_s^2 > c\} \wedge 1.$$
On $\{\sigma^2 \in {\mathcal D}_\infty (c)\}$, we have $T_c=1$, thus
\begin{align*}
\E_{{\mathcal D}_\infty(c)}\Big[\Big|\int_0^1 g(s) dX_s\Big|^{2p}\Big] & =\E\Big[\Big|\int_0^{T_c} g(s) dX_s\Big|^{2p}\mathbb{I}_{\sigma^2\in {\mathcal D}_\infty(c)}\Big] \\
& \leq  \E\Big[\Big|\int_0^{T_c}g(s) dX_s\Big|^{2p}\Big].
\end{align*}
By Burkholder-Davis-Gundy inequality (later abbreviated by BDG, for a reference see \cite{kar}, p. 166), we have
\begin{align*}
	I 	\leq
	\E\Big[\Big|
	\int_0^{T_c} g(s) dX_s\big|^{2p}\Big]^{1/2} 
	\lesssim
 	\E\Big[\Big|\int_0^{T_c} g^2(s)\sigma_s^2 ds\Big|^p\Big]^{1/2}
	\lesssim \|g \|_{L^2}^{p},
\end{align*}
where we used that $\sigma_s^2 \leq c$ for $s \leq T_c$. For the term $II$, note first that if
\begin{align*}
	\tilde g(s) := \sum_{j=1}^n \big(\tfrac 1n\sum_{l=j}^n
	g'\big(\tfrac ln\big)\big) \mathbb{I}_{[(j-1)/n,j/n)}(s),\;\;s\in [0,1],
\end{align*}
the process	$S_t=\int_0^{t\wedge T_c} \big(\tilde g (s)+g(s)\big)dX_s$, $t \in [0,1]$ 
is a martingale and
\begin{align*}
	\langle S \rangle_1
	=\sum_{j=1}^n \int_{(j-1)/n}^{j/n}
	\Big(\tfrac 1n\sum_{l=j}^n g'\big(\tfrac ln\big)
	-\int_s^1 g'(u) du\Big)^2 \mathbb{I}_{\{s\leq T_c\}}d\langle X \rangle_s.
\end{align*}
By summation by parts, we derive 
\begin{align*}
	II &= \E_{{\mathcal D}_\infty(c)}\big[|S_1|^{2p}\big]^{1/2}
	\lesssim
	\E\big[\langle S\rangle_{T_c}^{p}\big]^{1/2}
	\lesssim \mathfrak{R}_n^{p}(g).
\end{align*}
\end{proof}

We further need some analytical properties of pre-averaging functions. In the following $\lambda,$ and $\tilde\lambda$ always denote a pre-averaging function and its normalized version (in the sense of Definition \ref{preaverage function}). We set
\begin{align}
 \Lambda(s):=\int_s^2 \widetilde\lambda(u)du  \ \mathbb{I}_{\left[0,2\right]}\left(s\right)
  \label{eq.Lambda def}
\end{align}
and
\begin{align}
\overline \Lambda(s):=\Big(\big(\int_0^s\widetilde\lambda(u)du\big)^2+\big(\int_0^{1-s}\widetilde\lambda(u)du\big)^2\Big)^{1/2} \ \mathbb{I}_{\left[0,1\right]}\left(s\right).
\label{eq. bar Lambda def} 
\end{align}
Note that for $i=2,\ldots,m$
$$\|\Lambda\big(m\,\cdot-(i-2)\big)\|_{L^2[0,1]} = m^{-1/2}\|\Lambda\|_{L^2[0,2]}$$ and 
$$\|\overline\Lambda\big(m\,\cdot-(i-1)\big) \|_{L^2\left[0,1\right]} = m^{-1/2}.$$

\begin{lemma} \label{analytical properties}
For $m \leq n$, we have
$$\mathfrak{R}_n\big[\Lambda\big(m\,\cdot-(i-2)\big)\big] \lesssim n^{-1}$$
and for $i=2,\ldots,m$
$$\|\Lambda\big(m\,\cdot-(i-2)\big)\|_{L^2} = m^{-1/2}.$$
\end{lemma}
\begin{proof}
Recall the definition of $\mathfrak{R}_n$ given in (\ref{eq.Rdef}) and let 
\begin{align}
 j_n^*(r):=\max\{j:j/n\leq r/m\}.
 \label{eq.jnstardef}
\end{align}
Since $\widetilde \lambda$ is bounded, we have
\begin{align*}
	&\max_{\tfrac jn\in \big(\tfrac{i-2}{m},\tfrac im\big]}
	\sup_{s \in \big[\tfrac{j-1}{n},\tfrac jn \big]}
	\Big|\tfrac 1n \sum_{l=j}^{j_n^*(i)} \widetilde\lambda\big(m\tfrac ln-(i-2)\big)
	-\int_s^1 \widetilde\lambda\big(mu-(i-2)\big) du\Big| \\
	&\leq \max_{\tfrac jn\in \big(\tfrac{i-2}{m},\tfrac im\big]}
	\sup_{s \in \big[\tfrac{j-1}{n},\tfrac jn \big]}
		\Big|\int_s^{\left(j-1\right)/n} \widetilde\lambda\big(mu-(i-2)\big) du \Big| +\\
	&\max_{\tfrac jn\in \big(\tfrac{i-2}{m},\tfrac im\big]}
	\sum_{l=j}^{j_n^\star(i)}
	\Big|\frac 1n \widetilde\lambda\big(m\frac ln-(i-2)\big)
	-\int_{(l-1)/n}^{l/n} \widetilde\lambda\big(mu-(i-2)\big) du \Big|
	\lesssim n^{-1},
\end{align*}
whence the first part of the lemma. For the second part, we have to prove that 
$$\left\|\Lambda\right\|_{L^2\left[0,2\right]} = 1.$$
This readily follows from
\begin{align*}
	\left\|\Lambda\right\|_{L^2\left[0,2\right]}^2
	&=\int_0^1\big(\int_s^2 \widetilde\lambda\left(u\right) du\big)^2
	ds
	+
	\int_1^2\big(\int_s^2 \widetilde\lambda\left(u\right) du\big)^2 ds \\
	&=
	\int_0^1\big(\int_0^s \widetilde\lambda\left(u\right) du\big)^2
	ds
	+
	\int_0^1\big(\int_{1+s}^2 \widetilde\lambda\left(u\right) du\big)^2 ds
	\\
	&=
	\int_0^1\big(\int_0^s \widetilde\lambda\left(u\right) du\big)^2
	ds
	+
	\int_0^1\big(\int_{1-s}^2 \widetilde \lambda\left(u\right) du\big)^2 ds
	=  \|\overline\Lambda\|_{L^2\left[0,1\right]}^2.
\end{align*}
\end{proof}

\begin{lemma}  \label{lemma discretization bis}
Work under Assumption \ref{basic assumption} and let $\Lambda$ as in (\ref{eq.Lambda def}) with $\lambda$ as in Definition \ref{preaverage function}. Then, for $m\leq n,$ every $p \geq 1$ and $c >0$, we have 
\begin{align*}
&\E_{{\mathcal D}_\infty(c)}\Big[\Big|\sum_{i=2}^mg\big(\tfrac{i-1}{m}\big)\Big(\int_0^1 \Lambda  \big(ms-(i-2)\big)dX_s\Big)^2 \\
&\quad -\int_0^1\sum_{i = 2}^m g\big(\tfrac{i-1}{m}\big)\Lambda^2(ms-(i-2))d\langle X\rangle_s\Big|^p\Big] \lesssim \|g\|_{L^\infty}^p |\mathrm{supp}(g)|^{p/2} m^{-p/2},
\end{align*}
where $|\mathrm{supp}(g)|$ denotes the support length of $g$. 
\end{lemma}
\begin{proof}
In the same way as for Lemma \ref{lemma discretization}, we may (and will) assume that $X$ is a local martingale. For $i=2,\ldots, m$ and $t \in [0,1]$, set
\begin{align}
	H_{t,i}
	:=g\big(\tfrac{i-1}{m}\big)\Lambda\big(mt-(i-2)\big)\int_{(i-2)/m}^t\Lambda  \big(ms-(i-2)\big)dX_s \,\mathbb{I}_{\big(\tfrac{i-2}{m},\tfrac im\big]}(t).
	\label{Hti definition}
\end{align}
For a continuous semimartingale $M$ starting at zero, we have the integration by parts formula $M^2=\langle M \rangle +2\int MdM.$ Thus,
\begin{align}
&\sum_{i=2}^m g\big(\tfrac{i-1}{m}\big)\Big[\Big(\int_0^1 \Lambda  \big(ms-(i-2)\big)dX_s\Big)^2  \notag \\  
  &\quad\quad\quad\quad\quad\quad -\int_0^1\Lambda^2\big(ms-(i-2)\big)d\langle X\rangle_s \Big] \notag \\
&=  2\sum_{i = 2}^m \int_{(i-2)/m}^{i/m} H_{t,i}\,dX_t.
\label{integration by parts}
\end{align}
For $t\in [0,1]$, the process $\sum_{i = 2}^m H_{t,i}$ is continuous (because of $\Lambda(0)=\Lambda(2)=0$) and adapted, hence $\int_0^t \sum_{i = 2}^mH_{s,i}\,dX_s$ is a continuous local martingale. Applying BDG and the localisation argument of Lemma \ref{lemma discretization}, we obtain
\begin{align*}
& \E_{{\mathcal D}_\infty(c)}\big[\big|\int_0^{T_c}\sum_{i = 2}^mH_{t,i}\,dX_t\big|^p\big]  \\
 &\lesssim  \E\big[\big|\int_0^{T_c}\big(\sum_{i = 2}^mH_{t,i}\big)^2\,dt\big|^{p/2}\big] 
 \lesssim \E\big[\big|\int_0^{T_c}\sum_{i = 2}^mH_{t,i}^2\,dt\big|^{p/2}\big] \\
 &\lesssim \E\big[\big|m^{-1} \sum_{i=2}^m (H^\star_{i})^2\big|^{p/2}\big] 
 \lesssim |\mathrm{supp}(g)|^{p/2-1} m^{-1} \sum_{i = 2}^m \E\big[(H^\star_{i})^{p}\big],
\end{align*} 
where $H^\star_i:=\sup_{t \leq T_c}|H_{t,i}|$ and where we used that $t\leadsto H_{t,i}$ has compact support with length of order $m^{-1}$. The last estimate followed by H\"older inequality. By BDG again, we derive
\begin{align}
\E\big[(H_i^\star)^p\big] \lesssim &\big|g\big(\tfrac{i-1}{m}\big)\big|^p\E\Big[\sup_{t \leq 2/m} \Big|\int_{(i-2)/m \wedge T_c}^{((i-2)/m+t)\wedge T_c}\Lambda\big(ms-(i-2)\big)dX_s\Big|^p\Big]  \notag \\
 \lesssim & \big|g\big(\tfrac{i-1}{m}\big)\big|^p \E\Big[\Big(\int_{(i-2)/m \wedge T_c}^{T_c} \Lambda^2\big(ms-(i-2)\big)\sigma_s^2ds\Big)^{p/2}\Big] \notag \\
 \lesssim &  \big|g\big(\tfrac{i-1}{m}\big)\big|^p m^{-p/2}.
	\label{sup H bound}
\end{align}
The result follows.
\end{proof}

\begin{lemma} \label{properties Omega}
Work under Assumption \ref{basic assumption}. Let ${\mathcal B}^s_{\pi,\infty}(c)$ denote a Besov ball with $s>1/\pi$ and $c>0$.

In the same setting as in Lemma \ref{lemma discretization bis}, for every $p \geq 1$, we have
\begin{align*}
& \E_{{\mathcal B}^s_{\pi,\infty}(c)}\Big[\Big|\sum_{i = 2}^mg\big(\tfrac{i-1}{m}\big)\overline{X}_{i,m}^2-\int_0^1g(s)\sigma_s^2ds\Big|^p\Big] 
\lesssim \; \|g\|_{L^\infty}^pm^{-p/2}|\mathrm{supp}(g)|^{p/2}\\ &+|g|_{1,m}^p m^{-\min\{s-1/\pi,1\}p} +|g|_{\mathrm{var},m}^p m^{-p},
\end{align*} 
where
\begin{align}
|g|_{\mathrm{var},m}:=|g(0)+g(1)|+\sum_{i = 1}^m\sup_{s,t \in [(i-1)/m,i/m]}|g(t)-g(s)|.
\label{eq.varm}
\end{align}
\end{lemma}
\begin{proof}
Recall from Section \ref{preliminaries estimator} that
$$\overline{X}_{i,m}(\lambda):=\frac{m}{n}\sum_{\tfrac{j}{n}\in \big(\tfrac{i-2}{m},\tfrac{i}{m}\big]}\widetilde \lambda\big(m \tfrac{j}{n} -(i-2)\big)X_{j/n}.$$
Since $s>1/\pi$, the class ${\mathcal B}^s_{\pi,\infty}(c) \subset {\mathcal D}_\infty(c')$ for some $c'=c'(s,\pi,c)$. Therefore, by Lemma \ref{lemma discretization}, we have
\begin{align}
\E_{{\mathcal B}^s_{\pi,\infty}(c)}\Big[\Big|\overline{X}_{i,m}^2-\Big(\int_0^1\Lambda\big(ms-(i-2)\big)dX_s\Big)^2\Big|^p\Big] \lesssim m^{-p/2}n^{-p}
	\label{X bar approx}
\end{align}
since 
$$\mathfrak{R}_n\big[\Lambda\big(m\,\cdot-(i-2)\big)\big] \lesssim n^{-1}$$
by Lemma \ref{analytical properties}, $\|\Lambda\big(m\,\cdot-(i-2)\big)\|_{L^2}=m^{-1/2}$ and $m\leq n.$ By H\"older inequality it follows
\begin{align}
\E_{{\mathcal B}^s_{\pi,\infty}(c)}\Big[\Big|\sum_{i = 2}^mg\big(\tfrac{i-1}{m}\big)\overline{X}_{i,m}^2-\sum_{i = 2}^mg\big(\tfrac{i-1}{m}\big)\Big(\int_0^1\Lambda\big(ms-(i-2)\big)dX_s\Big)^2\Big|^p\Big]& \notag \\
\lesssim |\supp(g)|^{p-1}m^{p-1} \quad\quad\quad\quad\quad\quad\quad\quad\quad\quad\quad\quad\quad\quad\quad\quad\quad\quad\quad\quad\quad &\notag \\ \times \E_{{\mathcal B}^s_{\pi,\infty}(c)}\Big[\sum_{i = 2}^m \Big|g\big(\tfrac{i-1}{m}\big)\Big|^p \Big|\overline{X}_{i,m}^2-\Big(\int_0^1\Lambda\big(ms-(i-2)\big)dX_s\Big)^2\Big|^p\Big]& \notag \\
\lesssim  \|g\|_{L^\infty}^p m^{p/2}n^{-p}|\supp(g)|^{p},&
\label{used for ld prop again}
\end{align}
which can be further bounded by $ \|g\|_{L^\infty}^p m^{-p/2}|\supp(g)|^{p/2}.$ By Lemma \ref{lemma discretization bis}, we have
\begin{align*}
 &\E_{{\mathcal B}^s_{\pi,\infty}(c)}\Big[\Big| \sum_{i = 2}^mg \big(\tfrac{i-1}{m}\big)\Big(\int_0^1\Lambda\big(ms-(i-2)\big)dX_s\Big)^2 \\ - &\int_0^1\sum_{i = 2}^m g\big(\tfrac{i-1}{m}\big)\Lambda^2\big(ms-(i-2)\big)\sigma_s^2ds\Big|^p\Big] \lesssim \|g\|_{L^\infty}^pm^{-p/2}|\mathrm{supp}(g)|^{p/2},
\end{align*}
therefore by the triangle inequality
\begin{align}
\E_{{\mathcal B}^s_{\pi,\infty}(c)}\Big[\Big|\sum_{i = 2}^mg\big(\tfrac{i-1}{m}\big)\overline{X}_{i,m}^2-\int_0^1\sum_{i = 2}^m g\big(\tfrac{i-1}{m}\big)\Lambda^2\big(ms-(i-2)\big)\sigma_s^2ds\Big|^p\Big]& \nonumber \\
\lesssim \, \|g\|_{L^\infty}^pm^{-p/2}|\mathrm{supp}(g)|^{p/2}.& \label{first estimate}
\end{align}
We are going to force the function $\overline\Lambda$ in \eqref{first estimate}. To this end, note that
\begin{align}
& \sum_{i = 2}^m g\big(\tfrac{i-1}{m}\big)\Lambda^2\big(ms-(i-2)\big)  \nonumber\\
= & \sum_{i = 1}^m g\big(\tfrac{i}{m}\big)\Big(\Lambda^2\big(ms-(i-2)\big)+\Lambda^2\big(ms-(i-1)\big)\Big)\mathbb{I}_{\big(\tfrac{i-1}{m},\tfrac{i}{m}\big]}(s) \nonumber \\
+ & \sum_{i = 1}^m \Big(g\big(\tfrac{i-1}{m}\big)-g\big(\tfrac{i}{m}\big)\Big)\Lambda^2\big(ms-(i-2)\big)\mathbb{I}_{\big(\tfrac{i-1}{m},\tfrac{i}{m}\big]}(s) \nonumber\\
- & g(0)\Lambda^2\big(ms+1\big)\mathbb{I}_{\big(0,\tfrac{1}{m}\big]}(s)-g(1)\Lambda^2\big(ms-(m-1)\big)\mathbb{I}_{\big(1-\tfrac{1}{m},1\big]}(s). \label{second estimate}
\end{align}
Moreover, because of $\widetilde\lambda(u)=-\widetilde\lambda(2-u)$, we have $\Lambda^2(u)=\Lambda^2(2-u)$ and also
$\Lambda(0)=0$,
$$
\Lambda^2\big(ms-(i-2)\big)  = \big(\int_0^{1-(ms-(i-1))}\widetilde\lambda(u)du\big)^2, \quad \text{for} \ s\in \big(\tfrac{i-1}{m},\tfrac{i}{m}\big],$$
$$\Lambda^2\big(ms-(i-1)\big)  =  \big(\int_0^{ms-(i-1)}\widetilde\lambda(u)du\big)^2, \quad \text{for} \ s\in \big(\tfrac{i-1}{m},\tfrac{i}{m}\big].$$
This gives for $s \in \big(\tfrac{i-1}{m}, \tfrac{i}{m}\big]$, and $\bar \Lambda$ as in (\ref{eq. bar Lambda def})
\begin{equation} \label{switch omega 2 to omega 3}
\overline\Lambda^2\big(ms-(i-1)\big) = \Lambda^2\big(ms-(i-2)\big)+\Lambda^2\big(ms-(i-1)\big),
\end{equation}
and $0$ otherwise. From (\ref{second estimate}) it follows that on the event $\sigma^2 \in {\mathcal B}^s_{\pi,\infty}(c)$
\begin{align}
	&\Big|
	\int_0^1\sum_{i = 2}^m g\big(\tfrac{i-1}{m}\big)\Lambda^2\big(ms-(i-2)\big)\sigma_s^2ds
	 \notag \\ 
	&\quad\quad -\int_0^1\sum_{i = 1}^m g\big(\tfrac{i}{m}\big)\overline\Lambda^2\big(ms-(i-1)\big)\sigma_s^2ds\Big| 
	\lesssim |g|_{\text{var},m}m^{-1}.
	\label{first approx}
\end{align}

Finally, we have for $\sigma^2 \in {\mathcal B}^s_{\pi,\infty}(c)$ using $\|\overline\Lambda\|_{L^2}=1$
\begin{align}
&\Big|\int_0^1\sum_{i = 2}^m g\big(\tfrac{i-1}{m}\big)\Big(\overline\Lambda^2\big(ms-(i-1)\big)-\mathbb{I}_{\big(\tfrac{i-1}{m},\tfrac{i}{m}\big]}(s)\Big)\sigma_s^2ds\Big| 
	\notag \\
\leq &\; \Big|\int_0^1\sum_{i = 2}^m g\big(\tfrac{i-1}{m}\big)\overline\Lambda^2\big(ms-(i-1)\big)\big(\sigma_s^2-\sigma_{(i-1)/m}^2\big)ds\Big| 
	\notag \\
+ & \Big|\int_0^1\sum_{i = 2}^m g\big(\tfrac{i-1}{m}\big)\mathbb{I}_{\big(\tfrac{i-1}{m},\tfrac{i}{m}\big]}(s)\big(\sigma_s^2-\sigma_{(i-1)/m}^2\big)ds\Big| 
	\notag \\
\lesssim &\; m^{-\min\{s-1/\pi,1\}}|g|_{1,m},
	\label{second approx}
\end{align}
the last estimate coming from the Sobolev embedding ${\mathcal B}^s_{\pi,\infty} \subset {\mathcal B}^{s-1/\pi}_{\infty,\infty}$ which contains H\"older continuous functions of smoothness $\min\{s-1/\pi,1\}$. Since for $\sigma^2 \in {\mathcal B}^s_{\pi,\infty}(c)$
\begin{align}
 \Big|\int_0^1 \sum_{i = 2}^mg\big(\tfrac{i}{m}\big)\mathbb{I}_{\big(\tfrac{i-1}{m},\tfrac{i}{m}\big]}(s)\sigma_s^2ds-\int_0^1g(s)\sigma_s^2ds\Big| \lesssim m^{-1}|g|_{\mathrm{var},m},
	\label{third approx}
\end{align}
 the conclusion follows by combining (\ref{first estimate}), (\ref{first approx}), (\ref{second approx}) and (\ref{third approx}).
\end{proof}
\subsubsection*{Preliminaries: some estimates for the microstructure noise $\epsilon$}
We need some notation. Remember from \eqref{observations} that we observe
$$Z_{j,n}=X_{j/n}+a(j/n,X_{j/n})\eta_{j,n},\;\;j=0,\ldots,n$$
where the intensity of microstructure noise process $a_s:=a(s,X_s)$ and noise innovations $\eta_{j,n}$ satisfy Assumption \ref{microstructure noise assumption}. 
For a pre-averaging function $\lambda$, recall from \eqref{pre averaging} that we define
\begin{align}
\overline{\epsilon}_{i,m}:=\overline{\epsilon}_{i,m}(\lambda):=\frac{m}{n}\sum_{\tfrac{j}{n}\in\big(\tfrac{i-2}{m},\tfrac{i}{m}\big]}\widetilde \lambda\big(m\tfrac{j}{n}-(i-2)\big)\epsilon_{j,n}, \quad i=2,\ldots,m.
  \label{eq.epsilonbardef}
\end{align}

\noindent
Moreover, we will make several times use of Rosenthal's inequality for martingales (see \cite{hal}, p. 23). It states that for an $(\mathcal{F}_k)_k$-martingale $(M_k)_k$ and for $p\geq 0,$ there exists a universal constant $C_p$ only depending on $p,$ such that
\begin{align*}
&\E\Big[\max_{k=1,\ldots,n} |M_k|^p \Big] \\
&\leq
C_p \Big(\E\Big[\big(\sum_{k=0}^{n-1} \E\big[(M_{k+1}-M_k)^2 | \mathcal{F}_k \big]\big)^{p/2}\Big]+\E\Big[\max_{k\leq n} |M_{k}-M_{k-1}|^p\Big]\Big).
\end{align*}
For our proofs it will be sufficient to bound the maximum in the second term on the r.h.s. by the sum $\sum_{k=1}^n.$

\begin{lemma} \label{first noise lemma}
Work under Assumption \ref{basic assumption} and \ref{microstructure noise assumption}. Let ${\mathcal G}$ denote the $\sigma$-field generated by $(X_s, s\in [0,1])$. For every function $g:[0,1]\rightarrow \R$ and $p\geq 1$, we have
\begin{align*}
& \E\Big[\Big|\sum_{i = 1}^m g\big(\tfrac{i-1}{m}\big)\big(\overline{\epsilon}^2_{i,m}(\lambda)-\E\big[\overline{\epsilon}_{i,m}^2(\lambda)\,\big|\,{\mathcal G}\big]\big)\Big|^p\Big] \\
& \lesssim \; |g|_{2,m}^pm^{3p/2}n^{-p}+|g|_{p,m}^pm^{p+1}n^{-p}.
\end{align*}
\end{lemma}
\begin{proof}
In the following, we will decompose the sum in the previous inequality in an even and odd part. This allows us to treat sums of preaveraged values computed over disjoint intervals. In a first step, let us introduce the filtrations 
\begin{align*}
{\mathcal F}^{ \text{\tiny{even}}}_r & :=\sigma\big(\eta_{j,n}:\;j/n\leq 2r/m\big) \otimes
\sigma\big(X_s:s\leq 2r/m\big), \\
{\mathcal F}^{\text{\tiny{odd}}}_r  & := \sigma\big(\eta_{j,n}:\; j/n \leq (2r+1)/m\big)\otimes
\sigma\big(X_s:s\leq (2r+1)/m\big).
\end{align*} 
Straightforward calculations show that the partial sums $S_r^{\text{\tiny{even}}}:=\sum_{i=1}^{r}U_{2i}$ and $S_r^{\text{\tiny{odd}}}:=\sum_{i = 1}^{r}U_{2i+1}$ with
$$U_i:=g\big(\tfrac{i-1}{m}\big)\Big(\overline{\epsilon}_{i,m}^2-\tfrac{m^2}{n^2}\sum_{\tfrac jn \in\big(\tfrac{i-2}{m},\tfrac{i}{m}\big]}\widetilde \lambda^2\big(m\tfrac{j}{n}-(i-2)\big)a_{j/n}^2\Big)$$
form martingale schemes $(i=1,\ldots,r\leq \lfloor m/2\rfloor)$ with respect to ${\mathcal F}^{ \text{\tiny{even}}}_r$ and ${\mathcal F}^{ \text{\tiny{odd}}}_r$ respectively. Intuitively, $\overline \epsilon_{i,m} =O_P(m^{1/2}/n^{1/2})$ by (\ref{eq.epsilonbardef}). More precisely using Rosenthal's inequality,  we have, for every $p \geq 1$
\begin{align*}
&\E\Big[\Big|\overline{\epsilon}_{i,m}^2-\tfrac{m^2}{n^2}\sum_{\tfrac jn\in\big(\tfrac{i-2}{m},\tfrac{i}{m}\big]}\widetilde \lambda^2\big(m\tfrac{j}{n}-(i-2)\big)a_{j/n}^2\Big|^p\Big] \\
&\lesssim \,\E\big[|\overline{\epsilon}_{i,m}|^{2p}\big]+\|\widetilde \lambda\|_{L^\infty}^{2p}\|a\|_{L^\infty}^{2p}m^pn^{-p} \lesssim m^pn^{-p},
\end{align*}
using $\|a\|_{L^\infty}\lesssim 1.$ It follows that 
\begin{align}
	\E\big[|U_i|^p\big]\lesssim |g(\tfrac{i-1}m)|^p m^{p}n^{-p}.
	\label{Ui moment bound}
\end{align}
Analogous computations show that
$$\E\big[U_{2i}^2\,|\,{\mathcal F}_{i-1}^{\text{\tiny{even}}}\big] \leq g^2\big(\tfrac{2i-1}{m}\big) \E\big[\overline{\epsilon}_{2i,m}^4\,|\,{\mathcal F}_{i-1}^{\text{\tiny{even}}}\big] \lesssim g^2\big(\tfrac{2i-1}{m}\big)m^{2}n^{-2}.$$
Therefore, applying Rosenthal's inequality again, we obtain
$$\E\big[|S_{\lfloor m/2\rfloor}^{\text{\tiny{even}}}|^p\big] \lesssim |g|_{2,m}^pm^{3p/2}n^{-p}+|g|_{p,m}^pm^{p+1}n^{-p}.$$
Likewise, we obtain the same estimate for $\E\big[|S^{\text{\tiny{odd}}}_{\lfloor (m-1)/2\rfloor}|^p\big]$. The conclusion follows.
\end{proof}
\begin{lemma} \label{second noise lemma}
In the same setting as in Lemma \ref{first noise lemma}, we have, for every $c>0$ and $p \geq 1$
\begin{align*}
& \E_{{\mathcal D}_\infty(c)}\Big[\Big|\sum_{i = 1}^m g\big(\tfrac{i-1}{m}\big)\overline{X}_{i,m}(\lambda)\,\overline{\epsilon}_{i,m}(\lambda)\Big|^p\Big] \\
& \lesssim \;|g|_{p,m}^p\big(n^{-p/2}m+m^{3p/2+1}n^{-3p/2}\big)+|g|_{2,m}^p\big(m^{p/2}n^{-p/2}+m^{2p}n^{-3p/2}\big). 
\end{align*}
\end{lemma}
\begin{proof} By Assumption \ref{basic assumption} and the same localisation procedure as in the proof of Lemma \ref{lemma discretization}, up to losing some constant, we may (and will) assume that $X$ is a local martingale such that $|\sigma_s| \leq c$  almost-surely and subsequently work with $\E[\cdot]$ instead of $\E_{{\mathcal D}_\infty(c)}[\cdot]$. 

In the same way as for the proof of Lemma \ref{first noise lemma}, we define an ${\mathcal F}^{\text{\tiny{even}}}$-martingale by setting
$$S^{\text{\tiny{even}}}_r:=\sum_{i=1}^{r}g\big(\tfrac{2i-1}{m}\big)\overline{X}_{2i,m}(\lambda)\overline{\epsilon}_{2i,m}(\lambda)$$
and proceed for $S^{\text{\tiny{odd}}}$ analogously. By Rosenthal's inequality for martingales and Cauchy-Schwarz,
\begin{align*}
\E\big[\big|S_{\lfloor m/2 \rfloor}^{\text{\tiny{even}}}\big|^p\big]  \lesssim & m^{p/2}n^{-p/2}\E\Big[\Big|\sum_{i = 1}^{\lfloor m/2\rfloor}g^2\big(\tfrac{2i-1}{m}\big)\E\big[ \hspace{2pt} \overline{X}_{2i,m}^2(\lambda)\,|\,{\mathcal F}_{i-1}^{\text{\tiny{even}}}\big]\Big|^{p/2}\Big] \\
& + \sum_{i = 1}^{\lfloor m/2\rfloor}\big|g\big(\tfrac{2i-1}{m}\big)\big|^p\big(\E\big[|\overline{X}_{2i,m}(\lambda)|^{2p}\big])^{1/2}\big(\E\big[|\overline{\epsilon}_{2i,m}(\lambda)|^{2p}\big]\big)^{1/2}.
\end{align*}
Note that,
\begin{align*}
\E\big[\big|\overline{X}_{i,m}(\lambda)\big|^{2p}\big]
\lesssim &\;\E\Big[\Big|\tfrac{m}{n}\sum_{\tfrac{j}{n}\in\big(\tfrac{i-2}{m},\tfrac{i}{m}\big]}\widetilde\lambda\big(m\tfrac{j}{n}-(i-2)\big)(X_{j/n}-X_{(i-2)/m})\Big|^{2p}\Big]  \\
& + m^{2p}n^{-2p}\E\big[|X_{(i-2)/m}|^{2p}\big],
\end{align*}
where we used the fact that, by Riemann's approximation, we have 
\begin{align}
	\Big|\sum_{\tfrac{j}{n}\in \big(\tfrac{i-2}{m},\tfrac{i}{m}\big]}\widetilde\lambda\big(m\tfrac{j}{n}-(i-2)\big)\Big| \lesssim 1.
	\label{Riemann approx}
\end{align}
It follows that $\E\big[\big|\overline{X}_{i,m}(\lambda)\big|^{2p}\big]$ is less than
\begin{align}
 \|\widetilde\lambda\|_{L^\infty}^{2p}\E\Big[\sup_{s \leq 2/m}|X_{(i-2)/m+s}-X_{(i-2)/m}|^{2p}\Big]+m^{2p}n^{-2p}\E\big[|X_{(i-2)/m}|^{2p}\big]
 \label{eq.estimatebysup}
\end{align}
which in turn is of order
$\|\widetilde\lambda\|_{L^\infty}^{2p}m^{-p}+m^{2p}n^{-2p}$
thanks to the localization argument for $\sigma$. In a similar way, we obtain
$$\E\big[ \hspace{1pt} \overline{X}_{2i,m}^2(\lambda)\,\big|\,{\mathcal F}_{i-1}^{\text{\tiny{even}}}\big]\lesssim m^{-1}+m^2n^{-2}X_{(2i-2)/m}^2\leq m^{-1}+m^2n^{-2}\sup_s X_{s}^2.$$
Recall that $\E\big[|\overline{\epsilon}_{i,m}|^{2p}\big] \lesssim m^pn^{-p}$. Putting together these estimates, we infer that $\E\big[\big|S^{\text{\tiny{even}}}_{\lfloor m/2\rfloor}\big|^p\big]$ satisfies the desired bound. We proceed likewise for $S_{\lfloor (m-1)/2\rfloor}^{\text{\tiny{odd}}}$. The conclusion follows.
\end{proof}
\subsubsection*{Preliminaries: some estimates for the bias correction $\mathfrak{b}$}
We need some notation. Recall the bias correction defined in \eqref{bias correction}
$$\mathfrak{b}(\lambda, Z_\cdot)_{i,m} 
:=  \frac{m^2}{2n^2}\sum_{\tfrac{j}{n}\in \big( \tfrac{i-2}{m}, \tfrac{i}{m}\big]}\widetilde \lambda^2\big(m\tfrac{j}{n}-(i-2)\big)\big(Z_{j,n}-Z_{j-1,n}\big)^2.
$$
We plan to use the following decomposition
$$
\mathfrak{b}(\lambda, Z_\cdot)_{i,m} = \mathfrak{b}(\lambda, X_\cdot)_{i,m} 
+\mathfrak{b}(\lambda, \varepsilon_\cdot)_{i,m} +2\mathfrak{c}(\lambda,X_\cdot,\epsilon_\cdot)_{i,m},
$$
where
\begin{align*}
& \mathfrak{c}(\lambda,X_\cdot,\epsilon_\cdot)_{i,m} \\ 
& :=  \;
 \frac{m^2}{2n^2}\sum_{\tfrac{j}{n}\in \big( \tfrac{i-2}{m}, \tfrac{i}{m}\big]}\widetilde \lambda^2\big(m\tfrac{j}{n}-(i-2)\big)\big(X_{j/n}-X_{(j-1)/n}\big)\big(\epsilon_{j,n}-\epsilon_{j-1,n}\big).
\end{align*}
\begin{lemma} \label{first bias lemma}
Work under Assumption \ref{basic assumption} and \ref{microstructure noise assumption}. For every $p \geq 1$, we have
\begin{align*}
& \E\Big[\Big|\sum_{i = 2}^mg\big(\tfrac{i-1}{m}\big)\big(\mathfrak{b}(\lambda, \epsilon_\cdot)_{i,m}-\tfrac{m^2}{n^2}\sum_{\tfrac{j}{n} \in \big(\tfrac{i-2}{m},\tfrac{i}{m}\big]}\widetilde \lambda^2\big(m\tfrac{j}{n}-(i-2)\big)a_{j/n}^2\big)\Big|^p\Big] \\
&\lesssim \;|g|_{1,m}^pm^{3p}n^{-2p}+|g|_{2,m}^pm^{2p}n^{-3p/2}+|g|_{p,m}^{p}m^{2p}n^{-2p+1}. 
\end{align*}
\end{lemma}
\begin{proof}
By triangle inequality, we bound the error by a constant times 
$$m^{2p}n^{-2p}(I+II+III+IV),$$
where
\begin{align*}
I :=&\E\Big[\Big|\sum_{i = 2}^mg\big(\tfrac{i-1}{m}\big)\sum_j\widetilde\lambda^2\big(m\tfrac{j}{n}-(i-2)\big)a_{j/n}^2\big(\eta_{j,n}^2-1\big)\Big|^p\Big], \\
II:= &\E\Big[\Big|\sum_{i = 2}^mg\big(\tfrac{i-1}{m}\big)\sum_j\widetilde\lambda^2\big(m\tfrac{j}{n}-(i-2)\big)a_{\tfrac{j-1}{n}}^2\big(\eta_{j-1,n}^2-1\big)\Big|^p\Big], \\
III:=&\E\Big[\Big|\sum_{i = 2}^mg\big(\tfrac{i-1}{m}\big)\sum_j\widetilde\lambda^2\big(m\tfrac{j}{n}-(i-2)\big)\big(a_{\tfrac{j}{n}}^2-a_{\tfrac{j-1}{n}}^2\big)\Big|^p\Big], \\
IV:=&\E\Big[\Big|\sum_{i = 2}^mg\big(\tfrac{i-1}{m}\big)\sum_j\widetilde\lambda^2\big(m\tfrac{j}{n}-(i-2)\big)\epsilon_{j-1,n}\epsilon_{j,n}\Big|^p\Big], 
\end{align*}
where, as before, the sum in $j$ expands over $\big\{j/n\in \big((i-2)/m,i/m\big]\big\}$.\\

\noindent $\bullet$ {\bf The terms $I$ and $II$}. We only bound $I,$ the same subsequent arguments applying for the term involving $\eta_{j-1,n}$. Let $\mathcal{F}_j=\sigma(\eta_{k,n}:k\leq j) \otimes \sigma(X_s: s\leq 1).$ By Rosenthal's inequality for martingales,
\begin{align*}
I\lesssim &\sum_{j = 1}^n\Big(\sum_{i = 2}^m\big|g\big(\tfrac{i-1}{m}\big)\big|^p\mathbb{I}_{\big\{\tfrac{j}{n}\in \big(\tfrac{i-2}{m},\tfrac{i}{m}\big]\big\}}\Big)\E\big[\big|\big(\eta_{j,n}^2-1\big)\big|^p\Big] \\
&+  
\Big|\sum_{j = 1}^n\sum_{i = 2}^mg^2\big(\tfrac{i-1}{m}\big)\mathbb{I}_{\big\{\tfrac{j}{n}\in \big(\tfrac{i-2}{m},\tfrac{i}{m}\big]\big\}}\E\big[\big(\eta_{j,n}^2-1\big)^2\,\big|\,{\mathcal F}_{j-1}\big]\Big|^{p/2}, \\
\lesssim & |g|_{p,m}^pn+|g|_{2,m}^p n^{p/2}.
\end{align*}
where we used the fact that the functions $a$ and $\widetilde\lambda$ are bounded.\\

\noindent $\bullet$ {\bf The term $III$.} Recall the definition of $j_n^*(r)$ given in (\ref{eq.jnstardef}). Summing by parts, we have
\begin{align*}
& \sum_{\tfrac{j}{n}\in \big(\tfrac{i-2}{m},\tfrac{i}{m}\big]}\widetilde\lambda^2\big(m\tfrac{j}{n}-(i-2)\big)\big(a_{j/n}^2-a_{(j-1)/n}^2\big) \\
 & =  - \sum_{\tfrac{j}{n}\in \big(\tfrac{i-2}{m},\tfrac{i}{m}\big]}a_{(j-1)/n}^2\Big(\widetilde\lambda^2\big(m\tfrac{j}{n}-(i-2)\big)-\widetilde\lambda^2\big(m\tfrac{j-1}{n}-(i-2)\big)\Big)\\
 &\quad\quad+ a_{j_n^*(i)/n}^2\widetilde\lambda^2(m\tfrac{j_n^*(i)}n-(i-2))-a_{j_n^*(i-2)/n}^2\widetilde\lambda^2(m\tfrac{j_n^*(i-2)}n-(i-2)).
 \end{align*} 
Since $a$ is bounded and $\widetilde\lambda$ has finite variation, we infer
$$\Big|\sum_{i = 2}^mg\big(\tfrac{i-1}{m}\big) \sum_{\tfrac{j}{n}\in \big(\tfrac{i-2}{m},\tfrac{i}{m}\big]}\widetilde\lambda^2\big(m\tfrac{j}{n}-(i-2)\big)\big(a_{j/n}^2-a_{(j-1)/n}^2\big)\Big|^p\lesssim |g|_{1,m}^pm^p.$$
\noindent $\bullet$ {\bf The term $IV$.} We may split the sum with respect to $j$ in even and odd part. Proceeding as for $I$ and $II$, we readily obtain
$$IV \lesssim |g|_{2,m}^pn^{p/2}+|g|_{p,m}^pn.$$
\end{proof}
\begin{lemma} \label{second bias lemma}
In the same setting as in Lemma \ref{first bias lemma}, for every $c>0$, we have
$$\E_{{\mathcal D}_\infty(c)}\Big[\Big|\sum_{i = 2}^mg\big(\tfrac{i-1}{m}\big)\mathfrak{b}(\lambda, X_\cdot)_{i,m}\Big|^p\Big]\lesssim |g|_{1,m}^pm^{p}n^{-p}.$$
\end{lemma}
\begin{proof}
In the same way as in the proof of Lemma \ref{second noise lemma}, we may (and will) assume that $X$ is a local martingale and that $|\sigma_s^2|\leq c$ almost surely, working subsequently with $\E[\cdot]$ instead of $\E_{{\mathcal D}_\infty}(c)[\cdot]$. We readily obtain
\begin{align*}
& \E\Big[\Big|\sum_{i = 2}^mg\big(\tfrac{i-1}{m}\big)\mathfrak{b}(\lambda, X_\cdot)_{i,m}\Big|^p\Big] \\
&\lesssim \; m^{2p}n^{-2p}\E\Big[\Big|\sum_{i = 2}^m\big|g\big(\tfrac{i-1}{m}\big)\big|\sum_{\tfrac{j}{n}\in \big(\tfrac{i-2}{m},\tfrac{i}{m}\big]}(X_{j/n}-X_{(i-2)/m})^2\Big|^p\Big] \\
&\lesssim \; |g|_{1,m}^pm^{p}n^{-p}
\end{align*}
where we bound $|X_{j/n}-X_{(i-2)/m}|$ by the supremum over $|X_{s+(i-2)/m}-X_{(i-2)/m}|, \ s\leq 2/m$ and argue as in (\ref{eq.estimatebysup}).
\end{proof}

Let $M$ be a continuous, locally square integrable $\mathcal{F}$-martingale and $H$ some progressively measurable process. Then, for $0\leq s<t\leq 1$ $$\E\Big[\Big( \int_s^t H_u dM_u \Big)^2 | \mathcal{F}_s\Big]= \E\Big[ \int_s^t H_u^2 d\langle M\rangle_u  | \mathcal{F}_s\Big]$$ provided that $\E\big[ \int_0^1 H_u^2 d\langle M\rangle_u\big]<\infty.$
This fact will be referred to in the sequel as conditional It\^o-isometry (cf. \cite{kar}, Section 3.2 B).

\begin{lemma} \label{third bias lemma}
In the same setting as in Lemma \ref{first bias lemma}, for every $c>0$, we have
\begin{align*}
& \E_{{\mathcal D}_\infty(c)}\Big[\Big|\sum_{i = 2}^mg\big(\tfrac{i-1}{m}\big)\mathfrak{c}(\lambda, X_\cdot,\epsilon_\cdot)_{i,m}\Big|^p\Big] \\
&\lesssim  \, \big[|g|_{2,m}^p+|g|_{p,m}^p(n^{-p/2+1}+m^{-p/2+1})\big]m^{2p}n^{-2p}.
\end{align*}
\end{lemma}
\begin{proof} As in Lemmas \ref{second noise lemma} and \ref{second bias lemma}, we may (and will) assume that $X$ is a local martingale and that $|\sigma_s^2|\leq c$ almost surely, working subsequently with $\E[\cdot]$ instead of $\E_{{\mathcal D}_\infty}(c)[\cdot]$. It suffices then to bound
$$\E\Big[\Big|\sum_{i = 1}^mg\big(\tfrac{i-1}{m}\big)\sum_{\tfrac{j}{n}\in \big(\tfrac{i-2}{m}, \tfrac{i}{m}\big]}\widetilde\lambda^2\big(m\tfrac{j}{n}-(i-2)\big)(X_{j/n}-X_{(j-1)/n})\epsilon_{j,n}\Big|^p\Big].$$

Recall that $j_n^*(r)=\max\{j:j/n\leq r/m\}$ and let us introduce the filtrations
\begin{align*}
{\mathcal G}^{ \text{\tiny{even}}}_r & :=\sigma\big(\eta_{j,n}:\;j/n\leq 2r/m\big) \otimes
\sigma\big(X_s:s\leq j_n^*(2r)/n\big), \\
{\mathcal G}^{\text{\tiny{odd}}}_r  & := \sigma\big(\eta_{j,n}:\; j/n \leq (2r+1)/m\big)\otimes
\sigma\big(X_s:s\leq j_n^*(2r+1)/n \big).
\end{align*}
The process
$$S_r^{\text{\tiny{even}}}:=\sum_{i = 1}^{r}g\big(\tfrac{2i-1}{m}\big)\sum_{\tfrac{j}{n}\in \big(\tfrac{2i-2}{m}, \tfrac{2i}{m}\big]}\widetilde\lambda^2\big(m\tfrac{j}{n}-(2i-2)\big)(X_{j/n}-X_{(j-1)/n})\epsilon_{j,n}$$
is a ${\mathcal G}^{\text{\tiny{even}}}$-martingale and likewise for $S_r^{\text{\tiny{odd}}}$ defined similarly w.r.t. the filtration ${\mathcal G}^{\text{\tiny{odd}}}_r$. Moreover, on one hand
\begin{align*}
& \E\Big[\Big|g\big(\tfrac{i-1}{m}\big)\sum_{\tfrac{j}{n}\in \big(\tfrac{i-2}{m}, \tfrac{i}{m}\big]}\widetilde\lambda^2\big(m\tfrac{j}{n}-(i-2)\big)(X_{j/n}-X_{(j-1)/n})\epsilon_{j,n}\Big|^p\Big] \\
&\lesssim \; \big|g\big(\tfrac{i-1}{m}\big)\big|^p\Big(m^{-p/2}+\sum_{\tfrac{j}{n}\in \big(\tfrac{i-2}{m}, \tfrac{i}{m}\big]}\E\big[\big|(X_{j/n}-X_{(j-1)/n})\epsilon_{j,n}\big|^p\big]\Big) \\
&\lesssim \;\big|g\big(\tfrac{i-1}{m}\big)\big|^pm^{-1}(m^{-p/2+1}+n^{-p/2+1}),
\end{align*}
and on the other hand by conditional It\^o-isometry 
\begin{align*}
& \E\Big[\Big(g\big(\tfrac{2i-1}{m}\big)\sum_{\tfrac{j}{n}\in \big(\tfrac{2i-2}{m}, \tfrac{2i}{m}\big]}\widetilde\lambda^2\big(m\tfrac{j}{n}-(2i-2)\big)(X_{j/n}-X_{(j-1)/n})\epsilon_{j,n}\Big)^2\Big|\;{\mathcal G}_{i-1}^{\ev}\Big] \\
\lesssim & \;g^2\big(\tfrac{2i-1}{m}\big)\sum_{\tfrac{j}{n}\in \big(\tfrac{2i-2}{m}, \tfrac{2i}{m}\big]}\E\big[(X_{j/n}-X_{(j-1)/n})^2\;\big|{\mathcal G}_{i-1}^{\ev}\big] 
\lesssim  m^{-1}g^2\big(\tfrac{2i-1}{m}\big).
\end{align*}
Therefore, by Rosenthal's inequality for martingales, we infer
$$\E\big[\big|S_{\lfloor m/2\rfloor}^{\text{\tiny{even}}}\big|^p\big] \lesssim |g|_{p,m}^p(n^{-p/2+1}+m^{-p/2+1})+|g|_{2,m}^p.$$
We proceed likewise for $S^{\text{\tiny{odd}}}_{\lfloor (m-1)/2\rfloor}$ and the conclusion follows by incorporating the multiplicative term $m^{2p}n^{-2p}$ in front of the two error terms.
\end{proof}
\subsubsection*{Completion of proof of Theorem \ref{moment bounds}}
Since
$${\mathcal E}_m(h_{\ell k})=\sum_{i=2}^m h_{\ell k}\big(\tfrac{i-1}{m}\big)\big[ \hspace{2pt} \overline{Z}_{i,m}^2-b(\lambda,Z_\cdot)_{i,m}\big]$$
we plan to use the following decomposition
\begin{align}
	{\mathcal E}_m(h_{\ell k})-\langle \sigma^2, h_{\ell k}\rangle_{L^2} = I+II+III,
	\label{coeff diff comp}
\end{align}
with
\begin{align*}
I & := \sum_{i=2}^m h_{\ell k}\big(\tfrac{i-1}{m}\big) \hspace{1pt} \overline{X}_{i,m}^2-\langle \sigma^2, h_{\ell k}\rangle_{L^2}, \\
II & :=  \sum_{i=2}^m h_{\ell k}\big(\tfrac{i-1}{m}\big)\big[\overline{\epsilon}_{i,m}^2-\mathfrak{b}(\lambda,Z_\cdot)_{i,m}\big],\\
III & := 2\sum_{i=2}^m h_{\ell k}\big(\tfrac{i-1}{m}\big)\overline{X}_{i,m}\overline{\epsilon}_{i,m}.
\end{align*}
\noindent $\bullet$ {\bf The term $I$.} By Lemma \ref{properties Omega}, we have
\begin{align*}
\E_{{\mathcal B}^s_{\pi,\infty}(c)}[|I|^p] \lesssim &\;\|h_{\ell k}\|_{L^\infty}^pm^{-p/2}|\mathrm{supp}(h_{\ell k})|^{p/2}\\
&\;+ |h_{\ell k}|_{1,m}^p m^{-\min\{s-1/\pi,1\}p} +|h_{\ell k}|_{\mathrm{var},m}^p m^{-p}.
\end{align*}
Note that $\|h_{\ell k}\|_{L^\infty}\leq 2^{\ell /2}\|h\|_{L^\infty}$ and $|\mathrm{supp}(h_{\ell k})|^{p/2} \lesssim 2^{-\ell p/2}.$ By assumption, $h$ has a piecewise Lipschitz derivative. With (\ref{eq.varm}), we conclude 
\begin{align}
 |h_{\ell k}|_{\mathrm{var},m}\lesssim m^{1/2}.
  \label{eq.hellk_varm}
\end{align}
Thus, the term $I$ has the right order.

\smallskip

\noindent $\bullet$ {\bf The term $II$.} Applying successively Lemmas \ref{first noise lemma}, \ref{first bias lemma}, \ref{second bias lemma} and \ref{third bias lemma}, we derive using $m\leq n^{1/2}$
\begin{align*}
& \E_{{\mathcal B}^s_{\pi,\infty}}\big[|II|^p\big]
 \lesssim  |h_{\ell k}|_{1,m}^pm^pn^{-p} +|h_{\ell k}|_{2,m}^pm^{3p/2}n^{-p}+|h_{\ell k}|_{p,m}^pm^{p+1}n^{-p}.
\end{align*}
Since for $1\leq p\leq 2,$ by Jensen's inequality $\E_{{\mathcal B}^s_{\pi,\infty}}[|II|^p]\leq \E_{{\mathcal B}^s_{\pi,\infty}}[|II|^2]^{p/2}$ and for $p\geq 2,$ $|h_{\ell k}|_{p,m}^pm^{p+1}n^{-p}\lesssim 2^{l(p/2-1)}m^{p+1}n^{-p}\leq m^{3p/2}n^{-p},$ this term also has the right order.

\smallskip

\noindent $\bullet$ {\bf The term $III$.} Finally, by Lemma \ref{second noise lemma}, we have
\begin{align*}
& \E_{{\mathcal B}^s_{\pi,\infty}(c)}\Big[\Big|III\Big|^p\Big] \\
 &\lesssim \;|h_{\ell k}|_{p,m}^p\big(n^{-p/2}m+m^{3p/2+1}n^{-3p/2}\big)+|h_{\ell k}|_{2,m}^p\big(m^{p/2}n^{-p/2}+m^{2p}n^{-3p/2}\big), 
\end{align*}
which also has the right order by the same argument as above. The proof of Theorem \ref{moment bounds} is complete.
\subsection{Proof of Theorem \ref{deviation bounds}}
\subsubsection{Preliminary: a martingale deviation inequality}
If $(M_k)$ is a locally square integrable ${\mathcal F}_k$-martingale with $M_0=0$, we denote by $[M]_k=\sum_{i = 1}^k (\Delta M_i)^2$ with $\Delta M_i=M_i-M_{i-1}$ its quadratic variation and by $\langle M\rangle_k = \sum_{i = 1}^k \E\big[(\Delta M_i)^2\,|\,{\mathcal F}_{i-1}\big]$ its predictable compensator. We will heavily rely on the following result of Bercu and Touati \cite{ber}.
\begin{theorem}[Bercu and Touati \cite{ber}] \label{Bercu Touati}
Let $(M_k)$ be a locally square integrable martingale. Then, for all $x,y>0$, we have
$$\PP\big[|M_k|\geq x,\;[M]_k+\langle M\rangle_k \leq y\big] \leq 2 \exp\Big(-\frac{x^2}{2y}\Big).$$
\end{theorem}

\noindent

From Theorem \ref{Bercu Touati}, we infer the following estimate

\begin{lemma} \label{BT customization}
Let $(M_j)$ be a locally square integrable ${\mathcal F}_j$-martingale. Suppose that for $p\geq 1$ there is some deterministic sequence $(C_j)_j$ (with $j=j(m)$) and $\delta>0$ such that $\PP[\left\langle M\right\rangle_j > C_j(1+\delta)]\lesssim m^{-p}.$ If further for every $\kappa\geq 2$
\begin{equation} \label{assumption martingale}
\max_{i=1,\ldots, j}\E\big[|\Delta M_i|^\kappa\big] \lesssim 1,
\end{equation}
then,
$$\PP\Big[\big|M_j\big| > 2(1+\delta)\sqrt{C_j \ p \log m}\Big] \lesssim m^{-p}$$
provided $m^{q_0}\leq j \leq m$ for some $0 < q_0 \leq 1$ and there is an $\epsilon>0$ such that $C_j\gtrsim j^{1/2+\epsilon}.$
\end{lemma}

\begin{proof}
We have by Theorem \ref{Bercu Touati}
\begin{align*}
	&\PP\big[\big|M_j\big|\geq 2(1+\delta)\sqrt{C_j p\log m}\big] \\
	&\leq 2m^{-p}+\PP\big[[M]_j+\langle M\rangle_j > y, \ \langle M\rangle_j \leq C_j(1+\delta)\big]
	+\PP\big[\langle M\rangle_j>C_j(1+\delta)\big],
\end{align*}
with $y=2C_j(1+2\delta).$ Further we obtain
\begin{align*}
	\PP\big[[M]_j+\langle M\rangle_j > y, \ \langle M\rangle_j \leq C_j(1+\delta)\big]
	\leq \PP\big[[M]_j-\langle M\rangle_j > 2C_j \delta\big].
\end{align*}
Since $([M]_j-\langle M\rangle_j)$ is a $\mathcal{F}_j$-martingale it follows by Chebycheff's and Rosenthal's inequality for martingales and $\kappa\geq 2$
\begin{align*}
	\PP\big[[M]_j-\langle M\rangle_j > 2C_j \delta \big]
	&\lesssim 
	C_j^{-\kappa}
	\E\Big[\big|[M]_j-\langle M\rangle_j\big|^\kappa\Big] \\
	&\lesssim C_j^{-\kappa}
	\sum_{i=1}^j \E\big|\Delta M_i\big|^{2\kappa}
	+ C_j^{-\kappa}
	\E\Big|\sum_{i=1}^j \E\big[(\Delta M)_i^4 | \mathcal{F}_{i-1}\big]
	\Big|^{\kappa/2} \\
	&\lesssim C_j^{-\kappa}(j+j^{\kappa/2})\lesssim j^{-\epsilon \kappa},
\end{align*}
where we used H\"older's inequality
\begin{align*}
	\E\Big|\sum_{i=1}^j \E\big[(\Delta M)_i^4 | \mathcal{F}_{i-1}\big]
	\Big|^{\kappa/2} 
	\lesssim
	j^{\kappa/2-1}\sum_{i=1}^j  \E\Big[\E\big(|\Delta M_i|^{2\kappa} | \mathcal{F}_{i-1}\big)\Big]\lesssim j^{\kappa/2}.
\end{align*}
Choosing $\kappa:=q_0^{-1}p\epsilon^{-1} >2$, we finally obtain
$$\PP\big[[M]_j+\langle M\rangle_j > y, \ \langle M\rangle_j \leq C_j(1+\delta)\big] \lesssim j^{-p/q_0} \leq m^{-p}.$$
\end{proof}

\begin{lemma}
\label{large dev first term}
Work under the assumptions of Theorem \ref{deviation bounds} and suppose that $X$ has no drift, i.e. $b=0.$ If $\overline c=\overline c\left(s,\pi,c\right)$ is such that $\mathcal{B}_{\pi,\infty}^s\left(c\right) \subset \mathcal{D}_\infty \left(\overline c\right)$ then, we have for every fixed $\delta>0$
\begin{align*}
	&\PP\left[\left|\sum_{i=2}^m h_{\ell k}\left(\tfrac{i-1}m\right)
	\overline X_{i,m}^2\left(\lambda\right)-
	\left\langle \sigma^2,h_{\ell k}\right\rangle_{L^2}
	\right| \right. \\
	&\left.
	\quad\quad\quad\quad\quad>4\overline c\left(1+\delta\right)\sqrt{\tfrac{p\log m}{m}} \ \text{and} \ \sigma^2\in \mathcal{B}_{\pi,\infty}^s\left(c\right) \right]
	\lesssim m^{-p},
\end{align*}
provided
$$m^{-(s-1/\pi)}|h_{\ell k}|_{1,m} \lesssim m^{-1/2}.$$
\end{lemma}

\begin{proof}
Recall that $\Lambda\left(s\right)=\int_s^2 \widetilde\lambda\left(u\right) du$ and let $H_{t,i}$ be defined as in (\ref{Hti definition}), where $g$ is replaced by $h_{\ell k}.$ Using the integration by parts formula (\ref{integration by parts}) we bound the probability by $I+II+III$, with
\begin{align*}
	I:=
	&\PP\Big[\Big|\sum_{i=2}^m h_{\ell k}\big(\tfrac{i-1}m\big)
	\big(\overline X_{i,m}^2(\lambda)
	- \big(\int_0^1\Lambda(ms-\left(i-2\right))dX_s\big)^2\big)\Big|  \\
	& 
	\quad\quad\quad >
	\overline c \delta \sqrt{\tfrac{p\log m}m}
	\ \text{and} \ \sigma^2\in \mathcal{B}_{\pi, \infty}^s\left(c\right) \Big] \\
	II:=
	&\PP\Big[\Big|\sum_{i=2}^m \int_{0}^{1} H_{t,i} dX_t \Big|>
	2\overline c \left(1+\tfrac {\delta}{2}\right)  \sqrt{\tfrac{p\log m}m}
	\ \text{and} \ \sigma^2\in \mathcal{D}_\infty \left(\overline c\right) \Big] \\
	III:=
	&\PP\Big[\Big|\sum_{i=2}^m h_{\ell k} (\tfrac{i-1}m)
	\big(\int_0^1 \Lambda^2\left(ms-\left(i-2\right)\right)
	\sigma_s^2 ds
	- \left\langle \sigma^2,h_{\ell k}\right\rangle_{L^2}
	\big)\Big|   \\
	&  
	\quad\quad\quad >
	\overline c\delta \sqrt{\tfrac{p\log m}m}
	\ \text{and} \ \sigma^2\in \mathcal{B}_{\pi, \infty}^s\left(c\right) \Big].
\end{align*}

\noindent
Note that $\PP\left[X> t \ \text{and} \ B\right]
	= \E\left[\mathbb{I}_{\left\{X>t\right\} \cap B} \right] \leq t^{-p}\E\left[X^p \ \mathbb{I}_B\right],$ for $p\geq 0.$ Using $m\leq n^{1/2}$ and (\ref{used for ld prop again}) we find that $I$ can be bounded by any polynomial order of $1/m.$

\noindent
The term $II$ can be bounded further by $II\leq II_{\ev}+II_{\od},$ with
\begin{align*}
	II_{\ev/\od}:=
	&\PP\Big[\Big|\sum_{i=2, \ i \ \ev/\od}^m \int_{0}^{T_{\overline c}} H_{t,i} dX_t \Big|>
	\overline c\left(1+\tfrac {\delta}{2}\right)  \sqrt{\tfrac{p\log m}m}
	\Big].
\end{align*}
Since $h$ has support $[0,1]$, $h_{\ell k}(\tfrac{2i-1}m)\neq 0$ can happen only if 
\begin{align}
  \tfrac 12 (k2^{-\ell}m+1) \leq i \leq \tfrac 12 ((k+1)2^{-\ell}m+1).
  \label{eq.supportofmart}
\end{align}
We will treat the term $II_{\ev}$ only, since similar arguments apply for $II_{\od}.$ The process $M_r:= 2^{-\ell/2}m\sum_{i=1}^{r}\int_0^{T_{\overline c}} H_{t,2i} dX_t$ is a martingale with respect to the filtration $\mathcal{F}_r=\sigma\left(X_s: s\leq 2r/m\right)$ starting at $M_{\left\lfloor \left(k2^{-\ell}m+1\right)/2 \right\rfloor}=0.$ Recall that $H_{t,2i}$ vanishes outside $[2(i-1)/m,2i/m]$ and $\mathbb{I}_{\{T_{\overline c} \leq (2i-2)/m\}}$ is $\mathcal{F}_{i-1}$ measurable. Moreover, uniformly in $k, \ell,$ we obtain 
\begin{align}
 \frac 2m \sum_{i=1}^{\lfloor m/2\rfloor}
	h_{\ell k}^2\left(\tfrac{2i-1}m\right)= \|h_{\ell,k}\|_2^2+O(2^{\ell}/m)=1+O(m^{-q}).
  \label{eq.hlk riemann}
\end{align}
Therefore, Lemma \ref{analytical properties} and conditional It\^o-isometry yield
\begin{align*}
	\left\langle M\right\rangle_{\lfloor \tfrac 12 ((k+1)2^{-\ell}m+1) \rfloor} 
	&\leq
	2^{-\ell}m^2\overline c \sum_{i=1}^{\lfloor m/2\rfloor} \int_0^1 \E\left[H_{s\wedge T_{\overline c},2i}^2 |\mathcal{F}_{i-1}\right] ds \\
	&\leq 2^{-\ell-1}\overline c^2 \sum_{i=1}^{\lfloor m/2\rfloor}
	h_{\ell k}^2\left(\tfrac{2i-1}m\right)\leq 2^{-\ell}m  \ \tfrac 14 \overline c^2\left(1+\tfrac{\delta}2\right),
\end{align*}
where the last inequality follows for all $m\geq m_0(\delta)$ and $m_0(\delta)$ is fixed and independent of $\ell, k.$ Furthermore, by BDG and (\ref{sup H bound}), we bound
\begin{align*}
	\E\big[|\Delta M_i|^\kappa\big]
	&\lesssim 2^{-\ell\kappa/2}m^{\kappa}
	\E\Big[\big|\int_0^{1} H_{t,2i}\mathbb{I}_{\left[0,T_{\overline c}\right]}(t) dX_t\big|^{\kappa} \Big] \\
	&\lesssim 2^{-\ell\kappa/2}m^{\kappa}
	\E\Big[\big|\int_0^{1} H_{t\wedge T_{\overline c},2i}^2 dt\big|^{\kappa/2} \Big] \\
	&\lesssim 2^{-\ell\kappa/2}m^{\kappa/2}
	\E\Big[\sup_{t\leq 2/m}
	\big|H_{(t+2(i-1)/m)\wedge T_{\overline c},2i}\big|^\kappa\Big]
	\\
	&\lesssim 2^{-\ell \kappa/2} \big|h_{\ell k}\big(\tfrac{i-1}m\big)\big|^{\kappa} \lesssim 1
\end{align*}
uniformly over $i.$ Since the number of integers $i$ for which (\ref{eq.supportofmart}) holds is of order $m2^{-\ell},$ we may apply Lemma \ref{BT customization} for $j \sim m2^{-\ell},$ $C_j=2^{-\ell}m \tfrac 14\overline c^2$ and obtain $II_{\ev}\lesssim m^{-p}.$

\noindent
In the same way we bound $II_{\od}$ and thus obtain $II\lesssim m^{-p}.$

\noindent
In order to bound $III$ it follows from $m^{-(s-1/\pi)}|h_{\ell k}|_{1,m} \lesssim m^{-1/2},$ (\ref{first approx}), (\ref{second approx}), (\ref{third approx}), and (\ref{eq.hellk_varm}), that for sufficiently large $m$ on $\sigma^2\in \mathcal{B}_{\pi,\infty}^s\left(c\right)$
\begin{align*}
	\Big|\sum_{i=2}^m h_{\ell k} \left(\tfrac{i-1}m\right)
	\left(\int_0^1 \Lambda^2\left(ms-\left(i-2\right)\right)
	\sigma_s^2 ds
	- \left\langle \sigma^2,h_{\ell k}\right\rangle_{L^2}
	\right)\Big| 
	\leq 
	\overline c\delta \sqrt{\tfrac{p\log m}m}.
\end{align*}
This yields the conclusion.
\end{proof}


\begin{lemma}\label{large dev sec term}
Work under the assumptions of Theorem \ref{deviation bounds} and suppose that $X$ has no drift, i.e. $b=0.$ Then, we have for every fixed $\delta>0$
\begin{align*}
	&\PP\Big[\Big|\sum_{i=2}^m h_{\ell k}\left(\tfrac{i-1}m\right)
	\overline X_{i,m} \left(\lambda\right)\overline{\epsilon}_{i,m}(\lambda)
	\Big| \\
	&
	\quad\quad\quad\quad\quad>\sqrt{8\overline c}\left\|a\right\|_{L^\infty} \|\widetilde \lambda\|_{L^2} \left(1+\delta\right)\sqrt{\tfrac{p\log m}{m}} \ \text{and} \ \sigma^2\in \mathcal{B}_{\pi,\infty}^s\left(c\right) \Big]
	\lesssim m^{-p},
\end{align*}
where $\overline c\left(s,\pi,c\right)$ is such that $\mathcal{B}_{\pi,\infty}^s\left(c\right) \subset \mathcal{D}_\infty \left(\overline c\right).$
\end{lemma}

\begin{proof}
Let $\overline{X}_{i,m,T_{\overline c}}$ be defined as $\overline{X}_{i,m}$ with $X_{j/n}$ replaced by $X_{j/n \wedge T_{\overline c}}.$ Then by separating even and odd terms it suffices to show
\begin{align*}
	&\PP\Big[\Big|\sum_{i=2, \ i \ev}^m h_{\ell k}\left(\tfrac{i-1}m\right)
	\overline X_{i,m, T_{\overline c}} \left(\lambda\right)\overline{\epsilon}_{i,m}
	\Big| \\
	&
	\quad\quad\quad\quad\quad>\sqrt{2\overline c}\left\|a\right\|_{L^\infty} \|\widetilde \lambda\|_{L^2} \left(1+\delta\right)\sqrt{\tfrac{p\log m}{m}}\Big]
	\lesssim m^{-p}
\end{align*}
since the same argumentation can be done for the sum over odd $i.$ Similar as in the proof of Lemma \ref{large dev first term},  $M_r= n^{1/2}2^{-\ell/2}\sum_{i=1}^{2r} h_{\ell k}\left(\tfrac{2i-1}m\right) \overline X_{2i,m, T_{\overline c}} \overline{\epsilon}_{2i,m}$ defines a martingale with respect to the filtration $\mathcal{F}_r^{\ev},$ starting at $M_{\left\lfloor \left(k2^{-\ell}m+1\right)/2 \right\rfloor}=0.$
\begin{align*}
	&\left\langle M\right\rangle_{\lfloor \tfrac 12 ((k+1)2^{-\ell}m+1) \rfloor} 
	\leq n2^{-\ell}\sum_{i=1}^{\lfloor m/2\rfloor}
	h_{\ell k}^2(\tfrac{2i-1}m) \E[\overline X_{2i,m, T_{\overline c}}^2 \overline{\epsilon}_{2i,m}^2 |\mathcal{F}_{i-1}^{\ev}] \\
	&\leq n2^{-\ell}\left\|a\right\|_{L^\infty}^2
	\sum_{i=1}^{\lfloor m/2\rfloor}
	h_{\ell k}^2(\tfrac{2i-1}m) 
	\E[\overline X_{2i,m, T_{\overline c}}^2 |\mathcal{F}_{i-1}^{\ev}]\\
	&\quad\quad\quad\quad\quad\quad\quad\quad
	\times \tfrac{m^2}{n^2} \sum_{\tfrac jn \in \big (\tfrac{2i-2}m,\tfrac{2i}m\big ]}
	\widetilde \lambda^2(m\tfrac jn -(2i-2)).
\end{align*}
By the assumed piecewise Lipschitz continuity of $\lambda$ it follows
\begin{align}
	\tfrac mn \sum_{\tfrac jn \in \big (\tfrac{2i-2}m, \tfrac {2i}m\big ]}
	\widetilde \lambda^2(m\tfrac jn -(2i-2))
	= \|\widetilde \lambda\|_{L^2}^2+O\left(\tfrac mn\right),
	\label{riem approx L2}
\end{align}
uniformly in $i.$ Next, we will derive a bound for $\E[\overline X_{2i,m, T_{\overline c}}^2 |\mathcal{F}_{i-1}^{\ev}].$ Note that $\overline X_{2i,m, T_{\overline c}}= U_1+U_2,$ with
\begin{align*}
	U_1&:=\tfrac mn \sum_{\tfrac jn \in \big (\tfrac{2i-2}m, \tfrac {2i}m \big]}
	\Big(
	\sum_{l=j}^n \widetilde \lambda(m\tfrac ln -(2i-2))\Big)
	\big(X_{\tfrac jn \wedge T_{\overline c}}
	-X_{\tfrac {j-1}n \wedge T_{\overline c}\wedge \tfrac{2i-2}m}\big),\\
	U_2&:=X_{\tfrac{2i-2}m \wedge T_{\overline c}}
	\tfrac mn
	\sum_{\tfrac jn \in \big(\tfrac{2i-2}m, \tfrac {2i}m\big ]}
	\widetilde \lambda(m\tfrac jn -(2i-2)).
\end{align*}
Clearly, $\E[\overline X_{2i,m, T_{\overline c}}^2 |\mathcal{F}_{i-1}^{\ev}] = \E[U_1^2 |\mathcal{F}_{i-1}^{\ev}]+U_2^2.$ By conditional It\^o-isometry
\begin{align*}
	&\E\big[\big(X_{\tfrac jn \wedge T_{\overline c}}
	-X_{\tfrac {j-1}n \wedge T_{\overline c}\wedge \tfrac{2i-2}m}\big)
	\big(X_{\tfrac {j'}n \wedge T_{\overline c}}
	-X_{\tfrac {j'-1}n \wedge T_{\overline c}\wedge \tfrac{2i-2}m}\big )
	|\mathcal{F}_{i-1}^{\ev}\big]
	\leq \delta_{j,j'}\overline c \tfrac 1n \\
	&=\overline c\E\Big[\big(W_{\tfrac jn}-W_{\tfrac{j-1}n}\big)
	\big(W_{\tfrac {j'}n}-W_{\tfrac{j'-1}n}\big)\Big],
	\quad\quad\quad \text{for}\ \tfrac jn, \tfrac {j'}n \in (\tfrac {2i-2}m, \tfrac {2i}m],
\end{align*}
where $W$ is a standard Brownian motion and $\delta_{j,j'}$ denotes the Kronecker delta. Recall the definition of $j_n^*(r)$ given in (\ref{eq.jnstardef}) and define $c_j:=\sum_{l=j}^n \widetilde \lambda(m\tfrac ln -(2i-2)).$ We can bound
\begin{align*}
	&\E[U_1^2 |\mathcal{F}_{i-1}^{\ev}] \\
	&\leq \overline c \tfrac {m^2}{n^2} \Big(c_{1+j_n^*(2i-2)}^2 \tfrac{j_n^*(2i-2)}{n}+\sum_{\tfrac jn \in \big(\tfrac{2i-2}m, \tfrac{2i}m\big]}
    \tfrac {c_j^2}n\Big) \\
    &=
    \overline c \tfrac{m^2}{n^2}
    \E\Big[\Big(c_{1+j_n^*(2i-2)}W_{\tfrac{j_n^*(2i-2)}n}+\sum_{\tfrac jn \in \big(\tfrac{2i-2}m, \tfrac {2i}m\big]}
    c_j(W_{\tfrac jn}-W_{\tfrac {j-1}n})\Big)^2\Big] \\
    &=
    \overline c \E\Big[\Big( \tfrac mn \sum_{\tfrac jn \in \big(\tfrac{2i-2}m, \tfrac {2i}m\big]}
	\widetilde \lambda(m\tfrac jn -(2i-2)) W_{\tfrac jn}\Big )^2 \Big].
\end{align*}
Setting $X=W$ in (\ref{X bar approx}) and Lemma \ref{analytical properties} yield further
\begin{align*}
	\E[U_1^2 |\mathcal{F}_{i-1}^{\ev}]
	&\leq 
	\overline c \E\Big[\int_0^1 \Lambda^2\big(ms-(2i-2)\big) ds\Big]+O\big(m^{-1/2}n^{-1}\big) \\
	&=\overline c m^{-1}+O\big(m^{-1/2}n^{-1}\big)
\end{align*}
uniformly over $i.$ By using (\ref{Riemann approx}) we infer that there exists a constant $c_U$ such that $U_2^2\leq c_U\tfrac {m^2}{n^2}\sup_{s\leq T_{\overline c}} X_s^2.$ Choose $\delta'\leq \min(1,\tfrac{\delta}8 \min(\|\widetilde \lambda \|_{L^2}^2,1)).$ We find by Chebycheff inequality  that $\PP[\tfrac{m}{\overline c} U_2^2>\delta']\lesssim m^{-p}.$ With (\ref{eq.hlk riemann}), we obtain further for the predictable quadratic variation, sufficiently large $m$ and probability larger than $1-$const.$\times m^{-p}$
\begin{align*}
	&\left\langle M\right\rangle_{\lfloor \tfrac 12 ((k+1)2^{-\ell}m+1) \rfloor} 
	\\
	&\leq 2^{-\ell-1}m\left\|a\right\|_{L^\infty}^2
	\overline c \big(1+O(m^{-q})\big) \big(\|\widetilde \lambda\|_{L^2}^2+O(\tfrac mn)\big)
	\Big(1+ \tfrac{m}{\overline c} U_2^2  \Big) \\
	&\leq 2^{-\ell-1}m\left\|a\right\|_{L^\infty}^2
	\overline c (1+\delta') \big(\|\widetilde \lambda\|_{L^2}^2+\delta'\big)
	(1+ \delta'  ) \\
	&\leq 2^{-\ell-1}m\left\|a\right\|_{L^\infty}^2
	\overline c \|\widetilde \lambda\|_{L^2}^2
	(1+ \delta)
\end{align*}
or to state it differently
\begin{align*}
	\PP\Big[\left\langle M\right\rangle_{\lfloor \tfrac 12 ((k+1)2^{-\ell}m+1) \rfloor} 
	> 2^{-\ell-1}m\left\|a\right\|_{L^\infty}^2
	\overline c \|\widetilde \lambda\|_{L^2}^2(1+\delta)
	\Big]\lesssim m^{-p}.
\end{align*}
In the next step, we bound $\max_{i}\E[|\Delta M_i|^\kappa].$ In the proof of Lemma \ref{second noise lemma}, we already derived $\E[|\overline{X}_{i,m}(\lambda)|^{2\kappa}]\lesssim m^{-\kappa}$ and $\E[|\overline{\epsilon}_{i,m}|^{2\kappa}]\lesssim m^\kappa n^{-\kappa}.$ By the same arguments we obtain also $\E[|\overline{X}_{i,m,T_{\overline c}}(\lambda)|^{2\kappa}]\lesssim m^{-\kappa}.$ Therefore, it is easy to see that
\begin{align*}
	\max_{i}\E[|\Delta M_i|^\kappa]
	\lesssim 2^{-\ell\kappa/2}n^{\kappa/2}
	\big |h_{\ell k}(\tfrac{i-1}m)\big |^{\kappa}\E^{1/2}[|\overline{X}_{i,m,T_{\overline c}}(\lambda)|^{2\kappa}]\E^{1/2}[|\overline{\epsilon}_{i,m}|^{2\kappa}]\lesssim 1.
\end{align*}
Hence the assumptions of Lemma \ref{BT customization} are satisfied with $j\sim m2^{-\ell}$ and $C_j=2^{-\ell-1}m\left\|a\right\|_{L^\infty}^2
	\overline c \|\widetilde \lambda\|_{L^2}^2$ and the conclusion follows.
\end{proof}

\begin{lemma}\label{large dev third term}
Work under the assumptions of Theorem \ref{deviation bounds}. Let $\mathcal{G}$ denote the $\sigma$-field generated by $(X_s, s\in[0,1]).$ Then we have for every fixed $\delta>0$
\begin{align*}
	&\PP\Big[\Big|\sum_{i=2}^m h_{\ell k}\left(\tfrac{i-1}m\right)
	\big(\overline{\epsilon}_{i,m}^2(\lambda)
	-\E[\overline{\epsilon}_{i,m}^2(\lambda) |\mathcal{G}]\big)
	\Big| \\
	&
	\quad\quad\quad\quad\quad>4\left\|a\right\|_{L^\infty}^2 \|\widetilde \lambda\|_{L^2}^2 \left(1+\delta\right)\sqrt{\tfrac{p\log m}{m}}\Big]
	\lesssim m^{-p}.
\end{align*}
\end{lemma}

\begin{proof}
We show that
\begin{align*}
	&\PP\Big[\Big|\sum_{i=2, \ i \ev}^m h_{\ell k}\left(\tfrac{i-1}m\right)
	\big(\overline{\epsilon}_{i,m}^2(\lambda)
	-\E[\overline{\epsilon}_{i,m}^2(\lambda) |\mathcal{G}]\big)
	\Big| \\
	&
	\quad\quad\quad\quad\quad>2\left\|a\right\|_{L^\infty}^2 \|\widetilde \lambda\|_{L^2}^2 \left(1+\delta\right)\sqrt{\tfrac{p\log m}{m}}\Big]
	\lesssim m^{-p}
\end{align*}
and argue similar for the sum over $i$ odd. Let $\mathcal{F}_r^{\ev}, \ U_i$ and the martingale $S_r^{\ev}$ be defined as in the proof of Lemma \ref{first noise lemma} with $g$ replaced by $h_{\ell k}$. Now $h_{\ell k}(\tfrac{2i-1}m)\neq 0$ can happen only if $\tfrac 12 (k2^{-\ell}m+1) \leq i \leq \tfrac 12 ((k+1)2^{-\ell}m+1).$ In the following we will consider the martingale $M_r:= \tfrac{n}{m}2^{-\ell/2}S_r^{\ev}$ started at $M_{\left\lfloor \left(k2^{-\ell}m+1\right)/2 \right\rfloor}=0.$  We obtain
\begin{align*}
	\langle M \rangle_{\lfloor \tfrac 12 ((k+1)2^{-\ell}m+1) \rfloor} 
	\leq \tfrac{n^2}{m^2}2^{-\ell}\sum_{i=1}^{\lfloor m/2\rfloor} h_{\ell k}^2(\tfrac{2i-1}m)
	\E\Big[\big(\overline{\epsilon}_{2i,m}^2
	-\E[\overline{\epsilon}_{2i,m}^2 |\mathcal{G}]\big)^2|
	\mathcal{F}_{i-1}^{\ev}\Big].
\end{align*}
Elementary calculations and (\ref{riem approx L2}) show further that we may find a deterministic bound, i.e. uniformly in $i$
\begin{align*}
	&\E\Big[\big(\overline{\epsilon}_{i,m}^2
	-\E[\overline{\epsilon}_{i,m}^2 |\mathcal{G}]\big)^2|
	\mathcal{F}_{i-1}^{\ev}\Big] \\
	&=
	2\left\|a\right\|_{L^{\infty}}^4\big(\tfrac{m^2}{n^2}\sum_{\tfrac jn \in \big(\tfrac{i-2}m, \tfrac im\big]}
	\widetilde \lambda^2(m\tfrac jn-(i-2))\big)^2+O\big(\tfrac {m^3}{n^3}\big) \\
	&=2\tfrac{m^2}{n^2}\left\|a\right\|_{L^\infty}^4\|\widetilde \lambda\|_{L^2}^4+O\big(\tfrac {m^3}{n^3}\big).
\end{align*}
From this and (\ref{eq.hlk riemann}) we obtain for sufficiently large $m,$
\begin{align*}
	\langle M \rangle_{\lfloor \tfrac 12 ((k+1)2^{-\ell}m+1) \rfloor} 
	\leq m2^{-\ell}\left\|a\right\|_{L^{\infty}}^4\|\widetilde \lambda \|_{L^2}^4(1+\delta).
\end{align*}
By (\ref{Ui moment bound}), we infer $\E[|\Delta M_i|^\kappa] \lesssim 1.$ Applying Lemma \ref{BT customization} yields the conclusion.
\end{proof}

\subsubsection*{Completion of proof of Theorem \ref{deviation bounds}}

Let $I, II,$ and $III$ be defined as in (\ref{coeff diff comp}) and suppose that $X$ has no drift.

\noindent $\bullet$ {\bf The term $I$.} By Lemma \ref{large dev first term}, we have 
$$\PP\Big[|I|>4\overline{c}(1+\delta)\sqrt{\tfrac{p\log m}m} \ \text{and} \ \sigma^2 \in \mathcal{B}_{\pi,\infty}^s(c)\Big]\lesssim m^{-p}.$$

\noindent $\bullet$ {\bf The term $II$.} Applying Lemmas \ref{large dev third term}, \ref{first bias lemma}, \ref{second bias lemma} and \ref{third bias lemma}, we derive by Chebycheff's inequality and $|h_{\ell k}|^p_{p,m}\lesssim m^{p/2-1}$ , $p\geq 2$
$$\PP\Big[|II|>4\left\|a\right\|_{L^\infty}^2\|\widetilde\lambda\|_{L^2}^2
(1+\delta)\sqrt{\tfrac{p\log m}m} \ \text{and} \ \sigma^2 \in \mathcal{B}_{\pi,\infty}^s(c)\Big]\lesssim m^{-p}.$$

\noindent $\bullet$ {\bf The term $III$.} We find by Lemma \ref{large dev sec term}
$$\PP\Big[|III|>4\sqrt{2 \ \overline c} \ \|a\|_{L^\infty}\|\widetilde \lambda\|_{L^2}
(1+\delta)\sqrt{\tfrac{p\log m}m} \ \text{and} \ \sigma^2 \in \mathcal{B}_{\pi,\infty}^s(c)\Big]\lesssim m^{-p}.$$

\noindent
If the drift is non-zero, we can argue by a change of measure as in Lemma \ref{lemma discretization} and obtain with Assumption \ref{basic assumption}, $\E_{\sigma,b}[\mathbb{I}_{B_n}]\lesssim \E_{\sigma,0}[\mathbb{I}_{B_n}]^{(\rho-1)/\rho}.$ The proof of Theorem \ref{deviation bounds} is complete.


\subsection{Proof of Theorem \ref{resultlower}}

\noindent {\it Preliminaries.}
Let $(C,{\mathcal C})$ denote the space of continuous functions on  $[0,1]$, equipped with the norm of uniform convergence and its Borel $\sigma$-field ${\mathcal C}$. Let $(\Omega',{\mathcal F}', \PP')$ be another probability space rich enough to contain an infinite sequence of i.i.d. Gaussian random variables. On 
$(\widetilde \Omega, \widetilde {\mathcal F}):= (C\times C \times \Omega', {\mathcal C}\otimes {\mathcal C} \otimes {\mathcal F}')$ we construct a probability measure $\widetilde \PP$ as follows.
Let $(\sigma,\omega, \omega')$ denote a generic element of $\widetilde \Omega$.

We pick an arbitrary probability measure $\mu(d\sigma)$ on $(C,{\mathcal C})$, and we construct the measure $\PP_\sigma(d\omega)$ on $(C,{\mathcal C})$ such that, under $\PP_\sigma$, the canonical process $X$ on $C$ is a solution (in a weak sense for instance) to
$$X_t = X_0 + \int_0^t \sigma_s\, dW_s,$$
where $W$ is a standard Wiener process. We then set 
$$\widetilde \PP:=\mu(d\sigma) \otimes \PP_\sigma(d\omega) \otimes \PP'(d\omega').
$$
This space is rich enough to contain our model: indeed, by construction, any $\mu(d\sigma)$ will be such that, under $\mu$, we have Assumption \ref{basic assumption}. By constructing on $(\Omega',{\mathcal F},\PP')$ an i.i.d. Gaussian noise $(\epsilon_{j,n})$ for $j =0,\ldots, n$ with constant variance function $a^2>0$ for a given $a^2>0$, the space $\widetilde \Omega$ is rich enough to contain an additive Gaussian microstructure noise, independent of $X$, and we have Assumption \ref{microstructure noise assumption}. Consider next the statistical experiment 
$${\mathcal E}_n = \big(C\times \Omega', {\mathcal C}\otimes {\mathcal F}', (\PP^n_\sigma, \sigma \in {\mathcal D})\big),$$
where ${\mathcal D} \subset C$ and $\PP^n_\sigma$ is the law of the data $(Z_{j,n})$, conditional on $\sigma$. The probability $\mu(d\sigma)$ can be interpreted as a prior distribution for the ``true" parameter $\sigma$. Let us now introduce the statistical experiment ${\mathcal E}_n'$ generated by the observation of the Gaussian measure
$$Y_n = \sqrt{2\sigma}+a n^{-1/4}\dot B$$ where $\dot B$ is a Gaussian white noise, with same parameter space ${\mathcal D}$, but living on a possibly different space $\Omega''$. We denote by $\mathbb{Q}_\sigma^n$ the law of $Y_n$. 

\begin{proof}[Completion of proof] Let ${\mathcal D} = \mathcal{B}^s_{\pi,\infty}(c)$ denote a Besov ball such that $s-1/\pi > 0$. Then ${\mathcal D} \subset C$. Assume further that $\mu$ is such that $\mu\big[{\mathcal D}\big]=1$. Then Condition \eqref{first cond lb} is satisfied. Moreover, for any estimator $\widehat \sigma_n$ and any $c'>0$, we have, by Markov inequality
\begin{align} 
&n^{\alpha(s,p,\pi)/2}\widetilde \E\big[\|\widehat \sigma_n^2 -\sigma^2\|_{L^p([0,1])}\mathbb{I}_{\big\{\sigma^2\in {\mathcal B}^s_{\pi,\infty}(c)\big\}}\big]  \nonumber \\
\geq &\;c' \int_C \mu(d\sigma)\PP_\sigma^n\big[n^{\alpha(s,p,\pi)/2}\|\widehat \sigma_n^2 -\sigma^2\|_{L^p([0,1])}\geq c'\big] \label{the LB we want}
\end{align}
since $\mu\big[{\mathcal D}\big]=1$. By the result of Rei\ss\,\cite{rei2}, since $s-1/\pi > (1+\sqrt{5})/4$, we have that ${\mathcal E}_n$ and ${\mathcal E}_n'$ are asymptotically equivalent. This means that we can approximate $\PP_\sigma^n$ by $\mathbb{Q}_\sigma^n$ in variational norm, uniformly in $\sigma$, up to randomisation via a Markov kernel $K$ that does not depend on $\sigma$. More precisely, for any $\varepsilon >0$, we have
\begin{align}
\Big|\PP_\sigma^n\big[n^{\alpha(s,p,\pi)/2}& \|\widehat \sigma_n^2 - 
\sigma^2\|_{L^p([0,1])}\geq c'\big]  \nonumber \\ 
& - K\mathbb{Q}_\sigma^n\big[n^{\alpha(s,p,\pi)/2}\|\widehat \sigma_n^2 -\sigma^2\|_{L^p([0,1])}\geq c'\big]\Big| \leq \varepsilon \label{le cam approx}
\end{align}
as soon as $n$ is large enough, and where we use the notation
$$K\mathbb{Q}_n(dx) = \int_{\Omega''} K(y,dx)\mathbb{Q}_n(dy),\;\;x \in {\mathcal C}\times \Omega',\;\;y\in \Omega''.$$
Now, there exist $c'>0$ and $\delta'>0$ such that for any estimator $F$ in ${\mathcal E}'_n$, by picking $\mu(d\sigma)$ as the least favourable prior in order to obtain lower bounds over Besov classes, we have
\begin{equation} \label{LB white noise}
\int_C \mu(d\sigma)\mathbb{Q}_\sigma^n\big[n^{\alpha(s,p,\pi)/2}\|F-\sigma^2\|_{L^p([0,1])}\geq c'\big] \geq  \delta' >0 
\end{equation}
for large enough $n$. This follows from classical analysis of the white Gaussian noise model, see for instance \cite{HKPT} in the framework of Besov spaces. 
Let us extend further \eqref{LB white noise} to the class of {\it randomised decisions}, that is estimators of the form $F(\xi,\cdot)$, where $\xi$ is an auxiliary random variable, living on an auxiliary probability space with law $\nu(d\xi)$. Conditional on $\xi$, an arbitrary randomised decision $F(\xi,\cdot)$, can be viewed as an estimator, therefore, by \eqref{LB white noise},we also have
$$\int_C \mu(d\sigma)\mathbb{Q}_\sigma^n\big[n^{\alpha(s,p,\pi)/2}\|F(\xi,\cdot) -\sigma^2\|_{L^p([0,1])}\geq c'\big] \geq \delta'\;\;\nu(d\xi)-\text{a.s.}$$
for large enough $n$. Integrating an applying Fubini, we derive
$$\int_C \mu(d\sigma)\int\nu(d\xi)\mathbb{Q}_\sigma^n\big[n^{\alpha(s,p,\pi)/2}\|F(\xi,\cdot) -\sigma^2\|_{L^p([0,1])}\geq c'\big] \geq  \delta'.$$
Since $\nu$ and $F$ are arbitrary, it suffices then to identify the randomised decision $F(\xi,\cdot)$ with the estimator $\widehat \sigma_n$ in ${\mathcal E}_n$ transported into a random decision in ${\mathcal E}_n'$ with the Markov kernel $K$ appearing in \eqref{le cam approx}. We thus obtain 
\begin{equation} \label{le cam final}
\int_C \mu(d\sigma)K\mathbb{Q}_\sigma^n\big[n^{\alpha(s,p,\pi)/2}\|\widehat \sigma_n^2 -\sigma^2\|_{L^p([0,1])}\geq c'\big] \geq  \delta'
\end{equation}
for large enough $n$. Putting together \eqref{the LB we want}, \eqref{le cam approx} and \eqref{le cam final}, we finally obtain
$$n^{\alpha(s,p,\pi)/2}\widetilde \E\big[\|\widehat \sigma_n^2 -\sigma^2\|_{L^p([0,1])}\mathbb{I}_{\big\{\sigma^2\in {\mathcal B}^s_{\pi,\infty}(c)\big\}}\big] \geq \delta' - \varepsilon >0$$
for large enough $n$. The proof of Theorem \ref{resultlower} is complete.
\end{proof}

\section*{Acknowledgment}
The research of Axel Munk and Johannes Schmidt-Hieber was supported by DFG Grant FOR 916 and GK 1023. The research of Marc Hoffmann was supported by the Agence Nationale de la Recherche, Grant No. ANR-08-BLAN-0220-01. Parts of the presented work were developed within the PhD project \cite{sch}. We thank two referees and an editor for valuable comments improving the paper considerably.

\bibliographystyle{plain}       
\bibliography{refsPart1}           

\end{document}